\documentclass[10pt,a4paper]{amsart}
\usepackage{verbatim}

\usepackage[toc]{appendix}
\usepackage[T1]{fontenc}

\usepackage{graphicx}
\usepackage{enumerate}
\usepackage{amsmath,amsfonts,amssymb}
\usepackage{color}
\def\loc{\operatorname{loc}}
\usepackage{cite}
\definecolor{citation}{rgb}{0.11,0.67,0.84}
\definecolor{formula}{rgb}{0.1,0.2,0.6}
\definecolor{url}{rgb}{0.11,0.67,0.84}
\usepackage{pgf,tikz}
\usepackage{mathrsfs}
\usepackage{fancyhdr}
\usepackage{dutchcal}

\usepackage{dsfont}

\newcommand{\reqnomode}{\tagsleft@false}

\usepackage[colorlinks,pdfpagelabels,pdfstartview = FitH,bookmarksopen = true,bookmarksnumbered = true,urlcolor=url,linkcolor = formula,plainpages = false,hypertexnames = false,citecolor = citation] {hyperref}

\def\dx{\,{\rm d}x}

\def\dist{\,{\rm dist}}

\def\diam{\,{\rm diam}}

\allowdisplaybreaks
\makeatletter
\DeclareRobustCommand*{\bfseries}{%
  \not@math@alphabet\bfseries\mathbf
  \fontseries\bfdefault\selectfont
  \boldmath
}

\DeclareMathOperator*{\osc}{osc}

\makeatother

\newlength{\defbaselineskip}
\newcommand{\setlinespacing}[1]
           {\setlength{\baselineskip}{#1 \defbaselineskip}}

\newtheorem{theorem}{Theorem}
\newtheorem{corollary}{Corollary}[section]
\newtheorem{definition}{Definition}
\newtheorem{remark}{Remark}[section]
\newtheorem{lemma}{Lemma}[section]
\newtheorem{proposition}{Proposition}[section]
\numberwithin{equation}{section}
\allowdisplaybreaks[4]

\newcommand{\kk}{\kappa}

\newcommand{\ti}[1]{\tilde{#1}}
\newcommand{\mf}[1]{\mathfrak{#1}}

\newcommand\eps\varepsilon

\def\eqn#1$$#2$${\begin{equation}\label#1#2\end{equation}}

\newcommand{\be}{\begin{equation}}
\newcommand{\ee}{\end{equation}}

\newcommand{\rr}{\varrho}

\newcommand{\snr}[1]{\lvert #1\rvert}

\newcommand{\nr}[1]{\lVert #1 \rVert}

\newcommand{\uu}{\mathfrak{u}}

\newcommand{\N}{\mathbb{N}}

\def\name[#1, #2]{#1 #2}

\title[Free transmission problems]{Fully nonlinear free transmission problems with nonhomogeneous degeneracies}

\author[De Filippis]{Cristiana De Filippis}  \address{Cristiana De Filippis\\Dipartimento di Matematica "Giuseppe Peano", Universit\`a di Torino\\ Via Carlo Alberto 10, 10123 Torino, Italy} \email{\url{cristiana.defilippis@unito.it}}

\begin{document}

\subjclass[2020]{35A01, 35B65, 35J60, 35J70, 35R35 \vspace{1mm}} 

\keywords{Fully nonlinear degenerate equations, Double Phase problems, Free transmission problems\vspace{1mm}}

\thanks{{\it Acknowledgements.}\ This work is supported by the University of Turin via the project "Regolarit\'a e propriet\'a qualitative delle soluzioni di equazioni alle derivate parziali".
\vspace{1mm}}

\maketitle

\begin{abstract}
We prove existence and regularity results for free transmission problems governed by fully nonlinear elliptic equations with nonhomogeneous degeneracies.
 \end{abstract}
\vspace{3mm}
{\small \tableofcontents}

\setlinespacing{1.08}
\section{Introduction}\label{si}
In this paper we provide existence and regularity results for the free transmission problem
\begin{flalign}\label{eq}
\left[\snr{Du}^{p^{+}\mathds{1}_{\{u>0\}}+p_{-}\mathds{1}_{\{u<0\}}}+a(x)\mathds{1}_{\{u>0\}}\snr{Du}^{q}+b(x)\mathds{1}_{\{u<0\}}\snr{Du}^{s}\right]F(D^{2}u)=f(x)\qquad \mbox{in} \ \ \Omega,
\end{flalign}
that models anisotropic diffusion processes characterized by multiple degeneracy phenomena. In fact, the degeneracy law displayed in \eqref{eq} develops discontinuities along $\partial\{x\in \Omega\colon u(x)>0\}$ and $\partial\{x\in \Omega\colon u(x)<0\}$, and it is also influenced by the possible vanishing of the coefficients $a(\cdot)$, $b(\cdot)$.  The various regions where each degeneracy regime is in force are in part unknown a priori as they vary according to the sign of solutions and the transmission interface can be interpreted as a free boundary, but there is also a nonhomogeneous degeneracy variation corresponding to the zero sets of the modulating coefficients $\{x\in \Omega\colon a(x)=0\}$ and $\{x\in \Omega\colon b(x)=0\}$. Transmission problems are essentially related to the analysis of models involving different constitutive laws holding in separate subregions of the domain. The systematic study of such problems started with \cite{pi} and since then it has undergone an intensive development, see \cite{amte,css,hprs,ilsi,ol,ps,sc} and references therein. We mention in particular \cite{hprs}, where the authors consider the degenerate free transmission problem
\begin{eqnarray}\label{90}
\snr{Du}^{\theta_{1}\mathds{1}_{\{u>0\}}+\theta_{2}\mathds{1}_{\{u<0\}}}F(D^{2}u)=f(x)\qquad \mbox{in} \ \ \Omega,
\end{eqnarray}
where $\theta_{1},\theta_{2}$ are nonnegative constants and prove existence and optimal H\"older continuity for the gradient of solutions to the associated Dirichlet problem. Notice that the degeneracy law appearing in \eqref{90} is close to be homogeneous, in the sense that for any fixed point it behaves as a power. Another way of interpreting equation \eqref{90} is as an instance of fully nonlinear elliptic equation with discontinuous variable exponent 
\begin{eqnarray}\label{91}
\snr{Du}^{p(x)}F(D^{2}u)=f(x)\qquad \mbox{in} \ \ \Omega,
\end{eqnarray}
whose regularity theory is treated for the case of continuous exponents in \cite{bprt}. In sharp contrast with the models described so far, equation \eqref{eq} features a strongly anisotropic structure in which several nonhomogeneous phases coexist and switch according to the sign of solutions. Precisely, in correspondence of positive values of $u$ (resp. negative value of $u$) we see the nonhomogeneous degeneracy $[\snr{Du}^{p^{+}}+a(x)\snr{Du}^{q}]$ (resp. $[\snr{Du}^{p_{-}}+b(x)\snr{Du}^{s}]$). Degeneracies of Double Phase type have been introduced in \cite{de1}, where it is investigated the H\"older continuity of the gradient of solutions to fully nonlinear elliptic equations as
\begin{eqnarray}\label{92}
[\snr{Du}^{p}+a(x)\snr{Du}^{q}]F(D^{2}u)=f(x)\qquad \mbox{in} \ \ \Omega,
\end{eqnarray}
where $0\le p\le q$, $0\le a(\cdot)\in C(\Omega)$ and $f\in C(\Omega)\cap L^{\infty}(\Omega)$. This new model received lots of attention recently in the setting of free boundary problems, nonhomogeneous $\infty$-laplacian equations or obstacle problems, cf. \cite{bz1,bz2,bz}; while in \cite{frz} the authors carefully combine the approaches of \cite{bprt,de1} to derive local $C^{1,\alpha_{0}}$-regularity for viscosity solution of the fully nonlinear equation with variable exponents and nonhomogeneous degeneracy
\begin{eqnarray}\label{93}
[\snr{Du}^{p(x)}+a(x)\snr{Du}^{q(x)}]F(D^{2}u)=f(x)\qquad \mbox{in} \ \ \Omega.
\end{eqnarray}
These results also cover Multi-Phase equations with variable exponents, that are a generalization of \eqref{91}, \eqref{92} and \eqref{93}:
\begin{eqnarray}\label{94}
\left[\snr{Du}^{p(x)}+\sum_{\iota=1}^{\kk}a_{\iota}(x)\snr{Du}^{q_{\iota}(x)}\right]F(D^{2}u)=f(x)\qquad \mbox{in} \ \ \Omega.
\end{eqnarray}
In this framework, we introduce a new model for anisotropic free transmission problems which is essentially based on the alternance (according to the positivity of solutions) of degeneracies of type \eqref{92}, consider a Dirichlet problem governed by \eqref{eq} and prove that at least a solution exists. This is the content of
\begin{theorem}\label{t1}
Let $\Omega\subset \mathbb{R}^{n}$ be an open, bounded domain satisfying the uniform exterior sphere condition, assume \texttt{set} and let $g\in C(\partial \Omega)$. Then there exists a viscosity solution $u\in C(\bar{\Omega})$ to Dirichlet problem
\begin{flalign}\label{pd}
\begin{cases}
\ \left[\snr{Du}^{p_{u}(x)}+a(x)\mathds{1}_{\{u>0\}}\snr{Du}^{q}+b(x)\mathds{1}_{\{u<0\}}\snr{Du}^{s}\right]F(D^{2}u)=f(x)\quad &\mbox{in} \ \ \Omega\\
\ u=g\quad &\mbox{on} \ \ \partial \Omega,
\end{cases}
\end{flalign}
where it is $p_{u}(x):=p^{+}\mathds{1}_{\{u>0\}}(x)+p_{-}\mathds{1}_{\{u<0\}}(x).$
\end{theorem}
To prove Theorem \ref{t1} we first approximate a regularized variant of equation \eqref{pd}$_{1}$, obtained by relating the switch of degeneracies to the positivity of an arbitrary, globally continuous function, with a family of fully nonlinear Multi-Phase equations with continuous variable exponents cf. \eqref{94}, and prove local H\"older continuity estimates that are uniform with respect to the parameter of approximation and to the moduli of continuity of the variable exponents and of the coefficients, see Appendix \ref{vem}. Then we establish a comparison principle for the approximating Dirichlet problems, construct continuous supersolutions/subsolutions and and design a recursive procedure that will ultimately produce a solution to problem \eqref{pd} via Perron theorem. Once the matter of existence of solutions to \eqref{pd} has been settled, we turn to regularity. In this perspective, we have
\begin{theorem}\label{t2}
Assume \texttt{set} and let $u\in C(\Omega)$ be a viscosity solution of equation \eqref{eq}. Then there exists $\alpha_{0}\equiv \alpha_{0}(n,\lambda,\Lambda,p^{+},p_{-})\in (0,1)$ so that $u\in C^{1,\alpha_{0}}_{\loc}(\Omega)$. In particular, whenever $\Omega'\Subset \Omega$ is an open set it holds that
\begin{eqnarray}\label{60}
[Du]_{0,\alpha_{0};\Omega'}\le c(\texttt{data},\nr{u}_{L^{\infty}(\Omega)},\nr{f}_{L^{\infty}(\Omega)},\dist(\Omega',\partial\Omega)).
\end{eqnarray}
\end{theorem}
Theorem \ref{t2} is essentially based on the fact that viscosity solutions of \eqref{eq} turn out to be viscosity subsolution of 
\begin{flalign}\label{100}
\min\left\{F(D^{2}u),[\snr{Du}^{p^{+}}+a(x)\snr{Du}^{q}]F(D^{2}u),[\snr{Du}^{p_{-}}+b(x)\snr{Du}^{s}]F(D^{2}u)\right\}=\nr{f}_{L^{\infty}(\Omega)}
\end{flalign}
and viscosity supersolution to
\begin{flalign}\label{101}
\max\left\{F(D^{2}u),[\snr{Du}^{p^{+}}+a(x)\snr{Du}^{q}]F(D^{2}u),[\snr{Du}^{p_{-}}+b(x)\snr{Du}^{s}]F(D^{2}u)\right\}=-\nr{f}_{L^{\infty}(\Omega)}.
\end{flalign}
Combining this information with some delicate perturbation arguments, we are then able to transfer regularity from solutions of suitable homogeneous problems for which the Krylov-Safonov regularity theory is available, to solutions of \eqref{eq} and eventually conclude with \eqref{60}. We can safely conjecture that the strategies exposed here and in \cite{bprt,de1,hprs} provide a solid blueprint for studying also in the setting of free transmission problems models that are more anisotropic than \eqref{91}-\eqref{94} such as
\begin{eqnarray}\label{95}
\left[\snr{Du}\log(1+\snr{Du})+a(x)\snr{Du}^{q}\right]F(D^{2}u)=f(x)\qquad \mbox{in} \ \ \Omega
\end{eqnarray}
or, whenever $\varphi(\cdot)$ and $\psi(\cdot)$ are Orlicz functions:
\begin{eqnarray}\label{96}
\left[\varphi(\snr{Du})+a(x)\psi(\snr{Du})\right]F(D^{2}u)=f(x)\qquad \mbox{in} \ \ \Omega.
\end{eqnarray}
In \eqref{95}-\eqref{96}, $0\le a(\cdot)\in C(\Omega)$ is expected. Equations \eqref{91}-\eqref{94} and \eqref{95}-\eqref{96} are sophisticated examples of singular fully nonlinear elliptic equations, whose most celebrated prototype is
\begin{flalign}\label{classico}
\snr{Du}^{p}F(D^{2}u)=f\qquad \mbox{in} \ \ \Omega,
\end{flalign}
see e.g. \cite{art,bd4,is}. Several aspects of this class of partial differential equations are very well-known: comparison principle and Liouville type theorems \cite{bd1}, properties of eigenvalues and eigenfunctions \cite{bd5}, Alexandrov-Bakelman-Pucci estimates \cite{dafequ,imb}, Harnack inequalities \cite{dafequ1,imb} and regularity \cite{bd4,bd2,bd3,dkm,is}.
\subsection{Nonhomogeneous structures in the variational setting}
As one could expect, equations \eqref{91}-\eqref{94} and \eqref{95}-\eqref{96} have a variational counterpart. Although the study of nonhomogeneous structures in the fully nonlinear framework started very recently with \cite{bprt} for variable exponents and \cite{de1} in the Double Phase case, in the variational setting this field is object of intense investigation and the first results date back to the pioneering papers \cite{ma1,ma2}, where the author introduced the so-called functionals with $(p,q)$-growth, aimed at treating in a unified fashion some regularity aspects of several anisotropic functionals or equations with unbalanced polynomial growth. Later on, lots of efforts have been devoted to the analysis of specific nonautonomous models such as the $p(x)$-laplacian \cite{acmi, tac}:
\begin{eqnarray}\label{pxpx}
W^{1,1}_{\loc}(\Omega)\ni w\mapsto \int_{\Omega}\snr{Dw}^{p(x)} \ \dx,\qquad 1<\inf_{x\in \Omega}p(x)\le p(\cdot)\in C^{0,\alpha}(\Omega)
\end{eqnarray}
or the Double Phase energy \cite{bacomi,comi,demima}:
\begin{flalign}\label{dpdp}
&W^{1,1}_{\loc}(\Omega)\ni w\mapsto \int_{\Omega}\left[\snr{Dw}^{p}+a(x)\snr{Du}^{q}\right] \ \dx \\
&1<p\le q,\qquad \frac{q}{p}\le 1+\frac{\alpha}{n} \qquad 0\le a(\cdot)\in C^{0,\alpha}(\Omega),\nonumber
\end{flalign}
see also \cite{ciccio} for the analysis of nonhomogeneous problems including \eqref{pxpx}-\eqref{dpdp} and obstacle problems. A nontrivial extension of \eqref{dpdp} is the Multi-Phase energy \cite{bbo,deoh}:
\begin{flalign}\label{mpmp}
&W^{1,1}_{\loc}(\Omega)\ni w \mapsto \int_{\Omega} \left[|Dw|^p + \sum_{\iota=1}^{\kk} a_{\iota}(x)|Du|^{q_{\iota}}\right] \, \dx \\
&0\le a_{\iota}(\cdot)\in C^{0,\alpha_{\iota}}(\Omega),\qquad 1\le\frac{q_{\iota}}{p}\le 1+\frac{\alpha_{\iota}}{n},\qquad 1<p\le p_{1}\le \cdots\le p_{\kk},\nonumber 
\end{flalign}
which features several phase transitions in which the functional changes its ellipticity. This seems to be the right choice for modelling anisotropic free transmission problems and in fact its fully nonlinear version \eqref{94} is fundamental for the formulation of \eqref{eq}. A borderline version of \eqref{dpdp} is the following \cite{demi1}:
\begin{flalign*}
&W^{1,1}_{\loc}(\Omega)\ni w\mapsto \int_{\Omega}\left[\snr{Dw}\log(1+\snr{Dw})+a(x)\snr{Dw}^{q}\right] \ \dx\nonumber \\
&0\le a(\cdot)\in W^{1,d}(\Omega) \ \mbox{with} \ d>n,\qquad q<1+\frac{1}{n}-\frac{1}{d},
\end{flalign*}
which in nondivergence form becomes \eqref{95}. Other models inspired by \eqref{pxpx}-\eqref{dpdp} are the Double Phase energy with variable exponents \cite{rata,tac}, see \eqref{93}-\eqref{94} and the generalized Double Phase integral \cite{boh}, cf. \eqref{96}; we further refer to \cite{mira} for an account of the state of the art on this matter. The peculiarity of these functionals is that in the variational setting there is a strict interplay between the regularity of the $x$-depending coefficients and the regulairty of minimizers, therefore each of them has to be treated in a very specific way that takes into account the structure of the operator involved. Only recently a unified approach has been proposed in \cite{haok} in the framework of Musielak-Orlicz spaces. As already observed in \cite[Section 1]{de1}, there is a huge difference in the behavior of the nonhomogeneous structures listed above between the variational and the non-variational setting and this phenomenon is confirmed by Theorem \ref{t2} for anisotropic free transmission problems. In sharp contrast to what happens for instance with \eqref{dpdp}-\eqref{mpmp}, where a quantitative H\"older continuity (depending on the growth exponents) of the modulating coefficient $a(\cdot)$ is needed to get regular minima \cite{eslemi,fomami}, here the plain continuity of $a(\cdot)$ and $b(\cdot)$ suffices, cf. \eqref{ab}. In fact, to prove our regularity results, we just ask that the coefficients $a(\cdot)$, $b(\cdot)$ are continuous and no restriction on the size of the differences $0\le q-p^{+}$, $0\le s-p_{-}$ is imposed, see \eqref{pqs}. This makes Theorem \ref{t2} sharp from the viscosity theory viewpoint. 
\subsubsection*{Organization of the paper}
This paper is organized as follows. In Section \ref{pre} we display our notation, describe the main assumptions considered by Theorems \ref{t1}-\ref{t2} and recall some well-known results that will be needed in later on. In Section \ref{ex} we prove Theorem \ref{t1}, i.e. that there exist at least one solution to Dirichlet problem \eqref{pd}. In Section \ref{hr} we establish a uniform H\"older continuity result for solutions of suitable switched equations related to \eqref{eq}. Finally, Section \ref{ghgh} contains a description of the scaling properties of the viscosity differential inequalities \eqref{100}-\eqref{101}, a "harmonic" approximation lemma and the proof of Theorem \ref{t2}.
\section{Preliminaries}\label{pre}
We shall split this section in three parts: first, we display our notation, then we collect the main assumptions governing problem \eqref{eq}, and finally we report some well-known results on the theory of viscosity solutions to uniformly elliptic operators. 
\subsection{Notation}
In this paper, $\Omega\subset \mathbb{R}^{n}$, $n\ge 2$ is an open and bounded domain, the open ball of $\mathbb{R}^{n}$ centered at $x_{0}$ with positive radius $\rr$ is denoted by $B_{\rr}(x_{0}):=\left\{x\in \mathbb{R}^{n}\colon \snr{x-x_{0}}<\rr\right\}$. When not relevant, or clear from the context, we will omit indicating the center, $B_{\rr}\equiv B_{\rr}(x_{0})$. In particular, for $\rr=1$ and $x_{0}= 0$, we shall simply denote $B_{1}\equiv B_{1}(0)$. With $\mathcal{S}(n)$ we mean the space of $n\times n$ symmetric matrices. As usual, we denote by $c$ a general constant larger than one. Different occurrences from line to line will be still indicated by $c$ and relevant dependencies from certain parameters will be emphasized using brackets, i.e.: $c(n,p)$ means that $c$ depends on $n$ and $p$. With $z,\xi\in \mathbb{R}^{n}$, $\mu\in [0,1]$, $p,q,s\in [0,\infty)$ and $a(\cdot)$, $b(\cdot)$ being nonnegative functions, we define
\begin{eqnarray*}
\ell_{\mu}(z):=\sqrt{\mu^{2}+\snr{z}^{2}},\qquad p_{v}(x):=p^{+}\mathds{1}_{\{v>0\}}+p_{-}\mathds{1}_{\{v<0\}}
\end{eqnarray*}
and
\begin{flalign*}
&H_{q}(x,z;\xi):=\left[\snr{\xi+z}^{p^{+}}+a(x)\snr{\xi+z}^{q}\right],\qquad H_{s}(x,z;\xi):=\left[\snr{\xi+z}^{p^{-}}+b(x)\snr{\xi+z}^{s}\right]\nonumber \\
&\qquad \quad  H(x,v,z;\xi):=\left[\snr{\xi+z}^{p_{v}(x)}+a(x)\mathds{1}_{\{v>0\}}\snr{\xi+z}^{q}+b(x)\mathds{1}_{\{v<0\}}\snr{\xi+z}^{s}\right].
\end{flalign*}
When $\xi\equiv 0$, we shall simply write $H_{q}(x,z;0)\equiv H_{q}(x,z)$, $H_{s}(x,z;0)\equiv H_{s}(x,z)$ and $H(x,z;0)\equiv H(x,z)$. If $g\colon \Omega\to \mathbb{R}^{k}$ is any map, $U\subset \Omega$ is an open set and $\beta \in (0,1]$ is a given number we shall denote
\begin{flalign*}
[g]_{0,\beta;U}:=\sup_{x,y \in U; x\not=y}\frac{\snr{g(x)-g(y)}}{\snr{x-y}^{\beta}}, \qquad [g]_{0,\beta}:=[g]_{0,\beta;\Omega}.
\end{flalign*}
It is well known that the quantity defined above is a seminorm and when $[g]_{0,\beta;U}<\infty$, we will say that $g$ belongs to the H\"older space $C^{0,\beta}(U,\mathbb{R}^{k})$. We stress also that $g\in C^{1,\beta}(U,\mathbb{R}^{k})$ provided that
\begin{flalign*}
[g]_{1+\beta;U}:=\sup_{\rr>0,x\in U}\inf_{\xi\in \mathbb{R}^{n},\kappa\in \mathbb{R}}\sup_{y\in B_{\rr}(x)\cap U}\rr^{-(1+\beta)}\snr{g(y)-\xi\cdot y-\kappa}<\infty.
\end{flalign*}
Finally, $\mathbf{I}$ denotes the identity of $\mathbb{R}^{n\times n}$ and given any $n\times n$ matrix $A$, by $\texttt{tr}(A)$ we mean the trace of $A$, i.e., the sum of all its eigenvalues, by $\texttt{tr}(A^{+})$ the sum of all positive eigenvalues of $A$ and by $\texttt{tr}(A^{-})$ the sum of all negative eigenvalues of $A$. 
\subsection{On uniformly elliptic operators}\label{uniop} A map $G\in C( \Omega\times \mathbb{R}^{n}\times \mathcal{S}(n), \mathbb{R})$ is monotone if
\begin{flalign}\label{mon}
G(x,z,M)\le G(x,z,N)\quad \mbox{for all} \ \ M,N\in \mathcal{S}(n) \ \ \mbox{satisfy} \ \ M\ge N.
\end{flalign}
The $(\lambda,\Lambda)$-ellipticity condition for an operator $F\colon \mathcal{S}(n)\to \mathbb{R}$ prescribes that, whenever $A,B\in \mathbb{S}(n)$ are symmetric matrices with $B\ge 0$,
\begin{flalign}\label{ell}
\lambda \texttt{tr}(B)\le F(A)-F(A+B)\le \Lambda \texttt{tr}(B)
\end{flalign}
for and some fixed constants $0<\lambda\le \Lambda$. With this definition, $F(A):=-\texttt{tr}(A)$ is uniformly elliptic with $\lambda=\Lambda=1$ \cite{is}, so the usual Laplace operator is uniformly elliptic. Moreover, it is easy to see that, if $L$ is any fixed, positive constant, then the operator $F_{L}(M):=LF\left(\frac{1}{L}M\right)$ satisfies \eqref{ell} with the same constants $0<\lambda\le\Lambda$. Moreover, \eqref{ell} is also verified by the operator $\tilde{F}(M):=-F(-M)$, cf. \cite[Section 2.2]{de1}. In this framework, it is important to introduce the Pucci extremal operators $\mathcal{M}^{\pm}_{\lambda,\Lambda}(\cdot)$, which are, respectively, the maximum and the minimum of all the uniformly elliptic functions $F(\cdot)$ with $F(0)=0$. In particular they admit the compact form
\begin{flalign}\label{m-}
\mathcal{M}^{+}_{\lambda,\Lambda}(A)= -\Lambda \texttt{tr}(A^{-})-\lambda \texttt{tr}(A^{+})\quad \mbox{and}\quad \mathcal{M}^{-}_{\lambda,\Lambda}(A)=-\Lambda \texttt{tr}(A^{+})-\lambda \texttt{tr}(A^{-}).
\end{flalign}
We can give an alternative formulation of \eqref{ell} involving the Pucci extremal operators:
\begin{flalign}\label{elll}
\mathcal{M}^{-}_{\lambda,\Lambda}(B)\le F(A+B)-F(A)\le \mathcal{M}^{+}_{\lambda,\Lambda}(B),
\end{flalign}
that holds for all $A,B\in \mathcal{S}(n)$. Next, we turn our attention to equation
\begin{flalign}\label{211r}
G_{\xi}(x,Du,D^{2}u):=G(x,\xi+Du,D^{2}u)=0\quad \mbox{in} \ \ \Omega,
\end{flalign}
with $G(\cdot)$ continuous and satisfying \eqref{mon} and $\xi \in \mathbb{R}^{n}$ arbitrary vector. The concept of viscosity solution to \eqref{211r} can be explained as follows, cf. \cite{baim}:
\begin{definition}\label{def}
A lower semicontinuous function $v$ is a viscosity supersolution of \eqref{211r} if whenever $\varphi\in C^{2}(\Omega)$ and $x_{0}\in \Omega$ is a local minimum point of $v-\varphi$, then 
\begin{flalign*}
G_{\xi}(x_{0},D\varphi(x_{0}),D^{2}\varphi(x_{0}))\ge 0,
\end{flalign*}
while an upper semicontinuous function $w$ is a viscosity subsolution to \eqref{211r} provided that if $x_{0}$ is a local maximum point of $w-\varphi$, there holds
\begin{flalign*}
G_{\xi}(x_{0},D\varphi(x_{0}),D^{2}\varphi(x_{0}))\le 0.
\end{flalign*}
The map $u\in C(\Omega)$ is a viscosity solution of \eqref{211r} if it is a the same time a viscosity subsolution and a viscosity supersolution.
\end{definition}
In particular, a function $u\in W^{2,n}_{\loc}(\Omega)$ is a strong solution of \eqref{211r} if it verifies such identity almost everywhere in $\Omega$. Another important notion is the one of subjets and superjets \cite{baim}.
\begin{definition}\label{def2}
Let $v\colon \Omega\to \mathbb{R}$ be an upper semicontinuous function and $w\colon \Omega\to \mathbb{R}$ be a lower semicontinuous function.
\begin{itemize}
    \item A couple $(z,X)\in \mathbb{R}^{n}\times \mathcal{S}(n)$ is a superjet of $v$ at $x\in \Omega$ if
    \begin{flalign*}
    v(x+y)\le v(x)+z\cdot y+\frac{1}{2}Xy\cdot y+o(\snr{y}^{2}).
    \end{flalign*}
    \item A couple $(z,X)\in \mathbb{R}^{n}\times \mathcal{S}(n)$ is a subjet of $w$ at $x\in \Omega$ if
    \begin{flalign*}
    w(x+y)\ge w(x)+z\cdot y+\frac{1}{2}Xy\cdot y+o(\snr{y}^{2}).
    \end{flalign*}
    \item A couple $(z,X)\in \mathbb{R}^{n}\times \mathcal{S}(n)$ is a limiting superjet of $v$ ar $x\in \Omega$ if there exists a sequence $\{x_{j},z_{j},X_{j}\}\to_{j\to \infty}\{x,z,X\}$ such that $\{z_{j},X_{j}\}$ is a superjet of $v$ at $x_{j}$ and $v(x_{j})\to_{j\to \infty}v(x)$.
    \item A couple $(z,X)\in \mathbb{R}^{n}\times \mathcal{S}(n)$ is a limiting subjet of $w$ ar $x\in \Omega$ if there exists a sequence $\{x_{j},z_{j},X_{j}\}\to_{j\to \infty}\{x,z,X\}$ such that $\{z_{j},X_{j}\}$ is a subjet of $w$ at $x_{j}$ and $w(x_{j})\to_{j\to \infty}w(x)$.
\end{itemize}
\end{definition}
Now we are in position to present a variation on the celebrated Ishii-Lions lemma, \cite{baim,cil}.
\begin{proposition}\label{p1}
Let $v$ be an upper semicontinuous viscosity subsolution of \eqref{211r}, $w$ a lower semicontinuous viscosity supersolution of \eqref{211r}, $U\Subset \Omega$ an open set and $\psi \in C^{2}(U\times U)$. If $(\bar{x},\bar{y})\in U\times U$ is a local maximum point of $v(x)-w(y)-\psi(x,y)$, then, for any $\iota>0$ there exists a threshold $\hat{\delta}\equiv\hat{\delta}(\iota, \nr{D^{2}\psi})>0$ such that for all $\delta \in (0,\hat{\delta})$ we have matrices $X_{\delta},Y_{\delta}\in \mathcal{S}(n)$ such that
\begin{flalign*}
G_{\xi}(\bar{x},v(\bar{x}),\partial_{x}\psi(\bar{x},\bar{y}),X_{\delta})\le 0\le G_{\xi}(\bar{y},w(\bar{y}),-\partial_{y}\psi(\bar{x},\bar{y}),Y_{\delta}).
\end{flalign*}
and the inequality
\begin{flalign*}
-\frac{1}{\delta}\mathbf{I}\le\begin{bmatrix} X_{\delta} & 0 \\ 0 & -Y_{\delta}\end{bmatrix}\le D^{2}\psi(\bar{x},\bar{y})+\delta\mathbf{I}
\end{flalign*}
holds true.
\end{proposition}
So far, we have described the main features of equations governed by a continuous map $G(\cdot)$, while in the forthcoming sections we shall deal with problems with discontinuous degeneracies of the type
\begin{flalign}\label{d.1}
\left[\snr{\xi+Du}^{p_{v}(x)}+a(x)\mathds{1}_{\{v>0\}}\snr{\xi+Du}^{q}+b(x)\mathds{1}_{\{v<0\}}\snr{\xi+Du}^{s}\right]F(D^{2}u)=f(x)\qquad \mbox{in}\ \ \Omega,
\end{flalign}
where assumptions \texttt{set} is in force, $\xi\in \mathbb{R}^{n}$ is any vector and $v\in C(\bar{\Omega})$. In the light of the discussion in \cite[Section 2.2]{hprs}, we define a viscosity solution to \eqref{d.1} as follows.
\begin{definition}\label{def3}
Let $v\in C(\Omega)$ be a function, $\xi\in \mathbb{R}^{n}$ be a vector and assumptions \texttt{set} be in force. The map $u\in C(\Omega)$ is a viscosity solution to \eqref{d.1} if
\begin{itemize}
    \item in the set $\{x\in \Omega\colon v(x)>0\}\cap \{x\in \Omega\colon v(x)<0\}$ $u$ is a viscosity solution of \eqref{d.1};
    \item $u$ is a viscosity subsolution of 
    \begin{eqnarray*}
    \min\left\{F(D^{2}u),H_{q}(x,Du;\xi)F(D^{2}u),H_{s}(x,Du;\xi)F(D^{2}u)\right\}=f(x)
    \end{eqnarray*}
    in $\left\{x\in \Omega\colon v(x)=0\right\}$;
    \item $u$ is a viscosity supersolution of 
    \begin{eqnarray*}
    \max\left\{F(D^{2}u),H_{q}(x,Du;\xi)F(D^{2}u),H_{s}(x,Du;\xi)F(D^{2}u)\right\}=f(x)
    \end{eqnarray*}
    in $\left\{x\in \Omega\colon v(x)=0\right\}$.
\end{itemize}
\end{definition}
From now on, whenever we refer to a continuous viscosity solution of equations \eqref{eq} or \eqref{d.1} or to Dirichlet problem \eqref{pd}, we shall mean it in the sense of Definition \ref{def3}. The previous position can be justified by noticing that if we set $$\mathfrak{H}_{\xi}(x,v,z,X):=\left[H(x,v,z;\xi)F(X)-f(x)\right]$$ and introduce the semicontinuous envelopes $\mathfrak{H}_{\xi*}(\cdot)$ and $\mathfrak{H}_{\xi}^{*}(\cdot)$, then $u\in C(\Omega)$ is a viscosity solution of \eqref{d.1} if and only if it is a viscosity subsolution of $\mathfrak{H}_{\xi*}(x,v,Du,D^{2}u)=0$ and a viscosity supersolution of $\mathfrak{H}_{\xi}^{*}(x,v,Du,D^{2}u)=0$, cf. \cite{ccks,hprs}. We conclude this section by noticing that if $u\in C(\Omega)$ is a viscosity solution of equation \eqref{d.1}, then it is a viscosity subsolution to
\begin{flalign}\label{dif1}
\min\left\{F(D^{2}u),H_{q}(x,Du;\xi)F(D^{2}u),H_{s}(x,Du;\xi)F(D^{2}u)\right\}=\nr{f}_{L^{\infty}(\Omega)}\quad \mbox{in} \ \ \Omega
\end{flalign}
and a viscosity supersolution of
\begin{flalign}\label{dif2}
\max\left\{F(D^{2}u),H_{q}(x,Du;\xi)F(D^{2}u),H_{s}(x,Du;\xi)F(D^{2}u)\right\}=-\nr{f}_{L^{\infty}(\Omega)}\quad \mbox{in} \ \ \Omega.
\end{flalign}
This obsevation will be useful when proving regularity, see Sections \ref{hr}-\ref{ghgh} below.
\begin{remark}
\emph{If $u\in C(\Omega)$ is a viscosity solution of \eqref{d.1} in the sense of Definition \ref{def3}, then it is a viscosity subsolution/supersolution of variants of \eqref{dif1}/\eqref{dif2} having as a right-hand side term any constant larger or equal to $\nr{f}_{L^{\infty}(\Omega)}$.}
\end{remark}
\subsection{The homogeneous problem}
Viscosity solutions of the homogeneous problem
\begin{flalign}\label{probhom}
F(D^{2}h)=0\qquad \mbox{in} \ \ B_{1}(0)
\end{flalign}
will play a crucial role in the proof of the main results of this paper.
In fact, viscosity solutions of problem \eqref{probhom} have good regularity properties, as the next proposition shows. For a proof, we refer to \cite[Corollary 5.7]{caca}.
\begin{theorem}\label{rhar}
Assume that $F(\cdot)$ verifies \eqref{ell}-\eqref{assf} and $h\in C(B_{1}(0))$ be a viscosity solution of \eqref{probhom}. Then, there exist $\alpha\equiv\alpha(n,\lambda,\Lambda)\in (0,1)$ and $c\equiv c(n,\lambda,\Lambda)>0$ such that
\begin{flalign}\label{031}
\nr{h}_{C^{1,\alpha}(\bar{B}_{1/2}(0))}\le c\nr{h}_{L^{\infty}(B_{1}(0))}.
\end{flalign}
\end{theorem}
Theorem \ref{rhar} yields in particular that if $h\in C(B_{1}(0))$ is a viscosity solution to \eqref{probhom}, then it is $C^{1,\alpha}$-regular around zero, which means that for all $\rr\in (0,1)$ there exists a $\xi_{\rr}\in \mathbb{R}^{n}$ such that
\begin{flalign}\label{030}
\osc_{B_{\rr}}(h-\xi_{\rr}\cdot x)\le c(n,\lambda,\Lambda)\rr^{1+\alpha}.
\end{flalign}
Now fix $\sigma\in (0,1)$ so small that
\begin{flalign}\label{delta}
c\sigma^{\alpha}<\frac{1}{4},
\end{flalign}
where $c=c(n,\lambda,\Lambda)$ is the constant appearing in \eqref{031} and let $\xi_{\sigma}\in \mathbb{R}^{n}$ be the corresponding vector in \eqref{030}. According to the choice made in \eqref{delta}, estimate \eqref{030} reads as
\begin{flalign}\label{0331}
\osc_{B_{\sigma}(0)}(h-\xi_{\sigma}\cdot x)\le \frac{1}{4}\sigma, \qquad \mbox{with}\ \  \sigma\equiv\sigma(n,\lambda,\Lambda).
\end{flalign}
This will be useful later on.

\subsection{Main assumptions}\label{mass}
When dealing with equation \eqref{eq} or with Dirichlet problem \eqref{pd}, the following assumptions will be in force. As mentioned before, the set $\Omega\subset \mathbb{R}^{n}$ is an open and bounded domain with smooth boundary. 
The nonlinear operator $F(\cdot)$ is continuous and $(\lambda,\Lambda)$-elliptic in the sense of \eqref{ell}. Moreover
\begin{flalign}\label{assf}
F\in C(\mathcal{S}(n),\mathbb{R}), \qquad F(0)=0.
\end{flalign}
Concerning the non-homogeneous degeneracy term appearing in \eqref{eq}, we shall ask that the exponents $p^{+},p_{-},q,s$ satisfy 
\begin{eqnarray}\label{pqs}
0\le p^{+}\le q\qquad \mbox{and}\qquad 0\le p_{-}\le s,
\end{eqnarray}
the modulating coefficients $a(\cdot)$, $b(\cdot)$ are so that
\begin{eqnarray}\label{ab}
0\le a(\cdot)\in C(\Omega)\qquad \mbox{and}\qquad 0\le b(\cdot)\in C(\Omega)
\end{eqnarray}
and the forcing term $f$ verifies
\begin{eqnarray}\label{f}
f\in C(\Omega)\cap L^{\infty}(\Omega).
\end{eqnarray}
To simplify the notation, we shall collect the main parameters related to the problems under investigation in the shorthand $\texttt{data}:=\left(n,\lambda,\Lambda,p^{+},p_{-},q,s\right)$ and abbreviate the assumptions considered as $\texttt{set}:=\left(\eqref{ell},\eqref{assf},\eqref{pqs},\eqref{ab},\eqref{f}\right).$

\section{Existence of solutions}\label{ex}
In this section we prove the existence of a continuous viscosity solution to Dirichlet problem \eqref{pd}. To do so, we need to introduce a family of approximating problems, prove a comparison principle and then conclude via Perron method. For $\varepsilon\in (0,1)$, let $\{\phi_{\varepsilon}\}\subset C^{\infty}(\mathbb{R}^{n})$ be a sequence of radially symmetric, nonnegative mollifiers of $\mathbb{R}^{n}$, $v\in C(\bar{\Omega})$ be a continuous function, $\{\chi_{\varepsilon;v}^{+}\}:=\{\phi_{\varepsilon}*\mathds{1}_{\{v>0\}}\}\subset C^{\infty}_{\loc}(\Omega)$, $\{\chi_{\varepsilon;v}^{-}\}:=\{\phi_{\varepsilon}*\mathds{1}_{\{v<0\}}\}\subset C^{\infty}_{\loc}(\Omega)$,
\begin{flalign}\label{abe}
\begin{cases}
\ p_{\varepsilon;v}(x):=\varepsilon+p^{+}\chi_{\varepsilon;v}^{+}(x)+p_{-}\chi_{\varepsilon;v}^{-}(x)\\
\ a_{\varepsilon;v}(x):=(\varepsilon+a(x)\chi_{\varepsilon}^{+}(x))\\
\ b_{\varepsilon;v}(x):=(\varepsilon+b(x)\chi_{\varepsilon}^{-}(x)),
\end{cases}
\end{flalign}
for $p^{+}$, $p_{-}$, $a(\cdot)$, $b(\cdot)$ as in \eqref{pqs}-\eqref{ab}. By very definition, both the coefficients defined in \eqref{abe} are nonnegative and continuous in $\Omega$. With these positions at hand, set
\begin{eqnarray*}
\Omega\times \mathbb{R}^{n}\ni (x,z)\mapsto G_{\varepsilon}(x,z):=\left[\ell_{\varepsilon}(z)^{p_{\varepsilon;v}(x)}+a_{\varepsilon;v}(x)\ell_{\varepsilon}(z)^{q}+b_{\varepsilon;v}(x)\ell_{\varepsilon}(z)^{s}\right]
\end{eqnarray*}
and consider the equation
\begin{flalign}\label{eqe}
G_{\varepsilon}(x,Du_{\varepsilon})\left(\varepsilon u_{\varepsilon}+F(D^{2}u_{\varepsilon})\right)=f(x)\qquad \mbox{in} \ \ \Omega,
\end{flalign}
with $F(\cdot)$ as in \eqref{assf} and $f(\cdot)$ described by \eqref{f}. Let us prove a comparison principle for subsolutions and supersolutions of \eqref{eqe}.
\begin{lemma}\label{comp}
Let $\Omega\subset \mathbb{R}^{n}$ be an open bounded domain, assumptions \texttt{set} be in force, $w_{1}\in USC(\bar{\Omega})$ be a subsolution of \eqref{eqe} and $w_{2}\in LSC(\bar{\Omega})$ a supersolution of \eqref{eqe}. Then
\begin{eqnarray*}
w_{1}\le w_{2} \ \ \mbox{on} \ \ \partial \Omega \ \Longrightarrow \ w_{1}\le w_{2} \ \ \mbox{in} \ \ \Omega.
\end{eqnarray*}
\end{lemma}
\begin{proof}
By contradiction, assume that 
\begin{eqnarray}\label{1}
\omega_{0}:=\max_{x\in \Omega}\left(w_{1}(x)-w_{2}(x)\right)>0.
\end{eqnarray}
For $\sigma>0$, set
\begin{eqnarray*}
\Phi_{\sigma}(x,y):=w_{1}(x)-w_{2}(y)-\frac{\snr{x-y}^{2}}{2\sigma}
\end{eqnarray*}
and notice that, if $(x_{\sigma},y_{\sigma})\in \bar{\Omega}\times \bar{\Omega}$ is a point of maximum, i.e.: 
\begin{eqnarray}\label{2}
\max_{(x,y)\in \bar{\Omega}\times \bar{\Omega}}\Phi_{\sigma}(x,y)=\Phi_{\sigma}(x_{\sigma},y_{\sigma})\ge \omega_{0},
\end{eqnarray}
by \cite[Lemma 3.1]{cil} it is
\begin{eqnarray}\label{0}
\lim_{\sigma\to 0}\frac{\snr{x_{\sigma}-y_{\sigma}}^{2}}{\sigma}=0 \ \Longrightarrow \ \lim_{\sigma\to 0}\snr{x_{\sigma}-y_{\sigma}}=0.
\end{eqnarray}
Notice that $x_{\sigma},y_{\sigma}$ cannot both belong to $\partial \Omega$, otherwise $\Phi_{\sigma}(x_{\sigma},y_{\sigma})<0$, in contradiction with \eqref{1}-\eqref{2}. Then, at least one of them, say $x_{\sigma}$ must be in the interior of $\Omega$ and \eqref{0} forces also $y_{\sigma}$ to stay inside $\Omega$. We can then apply \cite[Theorem 3.2]{cil} to obtain that for all $\delta>0$ we have two symmetric matrices $X_{\delta},Y_{\delta}\in \mathcal{S}(n)$ so that $\left(\frac{x_{\sigma}-y_{\sigma}}{\sigma},X_{\delta}\right)$ is a superjet of $w_{1}$ in $x_{\sigma}$, $\left(\frac{x_{\sigma}-y_{\sigma}}{\sigma},Y_{\delta}\right)$ is a subjet of $w_{2}$ in $y_{\sigma}$ and the matrix inequality
\begin{eqnarray*}
\left(-\frac{1}{\delta}+c(n,\sigma)\right)\begin{bmatrix}
\mathbf{I} & 0 \\ 0 & \mathbf{I}
\end{bmatrix}\le \begin{bmatrix}
X_{\delta} & 0 \\ 0 & Y_{\delta} 
\end{bmatrix}\le\frac{3(1+\delta)}{\sigma}\begin{bmatrix}
\mathbf{I} & -\mathbf{I} \\ -\mathbf{I} & \mathbf{I}
\end{bmatrix}
\end{eqnarray*}
holds, therefore, testing against the couple $(\xi,\xi)\in \mathbb{R}^{2n}$ we get
\begin{eqnarray}\label{3}
2\left(-\frac{1}{\delta}+c(n,\sigma)\right)\snr{\xi}^{2}\le \langle(X_{\delta}-Y_{\delta})\xi,\xi\rangle\le 0 \ \Longrightarrow \ Y_{\delta}\ge X_{\delta}.
\end{eqnarray}
We can then recover the viscosity inequalities
\begin{eqnarray*}
\left\{
\begin{array}{c}
\displaystyle 
\ G_{\varepsilon}\left(x_{\sigma},\frac{x_{\sigma}-y_{\sigma}}{\sigma}\right) \left(\varepsilon w_{1}(x_{\sigma})+F(X_{\delta})\right)\le f(x_{\sigma}) \\[17pt]\displaystyle
\ G_{\varepsilon}\left(y_{\sigma},\frac{x_{\sigma}-y_{\sigma}}{\sigma}\right)\left(\varepsilon w_{2}(y_{\sigma})+F(Y_{\delta})\right)\ge f(y_{\sigma}),
\end{array}
\right.
\end{eqnarray*}
and subtract the second from the first to get
\begin{eqnarray}\label{66}
\frac{f(x_{\sigma})}{G_{\varepsilon}\left(x_{\sigma},\frac{x_{\sigma}-y_{\sigma}}{\sigma}\right)}-\frac{f(x_{\sigma})}{G_{\varepsilon}\left(y_{\sigma},\frac{x_{\sigma}-y_{\sigma}}{\sigma}\right)}&\ge&\varepsilon(w_{1}(x_{\sigma})-w_{2}(y_{\sigma}))+F(X_{\delta})-F(Y_{\delta})\nonumber \\
&\stackrel{\eqref{ell}}{\ge} & \varepsilon(w_{1}(x_{\sigma})-w_{2}(y_{\sigma}))+\lambda\texttt{tr}(Y_{\delta}-X_{\delta})\nonumber \\
& \stackrel{\eqref{3}}{\ge} & \varepsilon(w_{1}(x_{\sigma})-w_{2}(y_{\sigma}))\stackrel{\eqref{2}}{\ge}\varepsilon\omega_{0}.
\end{eqnarray}
At this point, recall that for any $\iota_{0}>0$ there exists a constant $c\equiv c(\iota_{0})$ such that for all $t\ge 0$, $l, m\ge 0$ it holds that $\snr{t^{l}-t^{m}}\le c\snr{l-m}(1+t^{(1+\iota_{0})\max\{l,m\}})$, so choosing $\iota_{0}:=\frac{\varepsilon}{16(p^{+}+p_{-}+1)}$ we see that
\begin{eqnarray*}
\begin{cases}
\ \iota_{0}\max\{p_{\varepsilon;v}(x_{\sigma}),p_{\varepsilon;v}(y_{\sigma})\}-\min\{p_{\varepsilon;v}(x_{\sigma}),p_{\varepsilon;v}(y_{\sigma})\}<-\frac{15\varepsilon}{16(p^{+}+p_{-}+1)}\\
\ \snr{\iota_{0}\max\{p_{\varepsilon;v}(x_{\sigma}),p_{\varepsilon;v}(y_{\sigma})\}-\min\{p_{\varepsilon;v}(x_{\sigma}),p_{\varepsilon;v}(y_{\sigma})\}}\le 4(p^{+}+p_{-}+1)
\end{cases}
\end{eqnarray*}
and so
\begin{eqnarray}\label{80}
\mathcal{L}(\varepsilon,\sigma)&:=&\ell_{\varepsilon}\left(\frac{x_{\sigma}-y_{\sigma}}{\sigma}\right)^{-(p_{\varepsilon;v}(x_{\sigma})+p_{\varepsilon;v}(y_{\sigma}))}\left| \ \ell_{\varepsilon}\left(\frac{x_{\sigma}-y_{\sigma}}{\sigma}\right)^{p_{\varepsilon;v}(x_{\sigma})}- \ell_{\varepsilon}\left(\frac{x_{\sigma}-y_{\sigma}}{\sigma}\right)^{p_{\varepsilon;v}(y_{\sigma})}\ \right|\nonumber \\
&\le&\snr{p_{\varepsilon;v}(x_{\sigma})-p_{\varepsilon;v}(y_{\sigma})}\ell_{\varepsilon}\left(\frac{x_{\sigma}-y_{\sigma}}{\sigma}\right)^{-(p_{\varepsilon;v}(x_{\sigma})+p_{\varepsilon;v}(y_{\sigma}))}\left[1+\ell_{\varepsilon}\left(\frac{x_{\sigma}-y_{\sigma}}{\sigma}\right)^{(1+\iota_{0})\hat{p}_{\varepsilon;\sigma}}\right]\nonumber \\
& \le& \snr{p_{\varepsilon;v}(x_{\sigma})-p_{\varepsilon;v}(y_{\sigma})}\left[\varepsilon^{-6\max\{p^{+},p_{-},1\}}+\ell_{\varepsilon}\left(\frac{x_{\sigma}-y_{\sigma}}{\sigma}\right)^{\iota_{0}\hat{p}_{\varepsilon;\sigma}-\ti{p}_{\varepsilon;\sigma}}\right]\nonumber \\
&\le&\frac{2\snr{p_{\varepsilon;v}(x_{\sigma})-p_{\varepsilon;v}(y_{\sigma})}}{\varepsilon^{6(p^{+}+p_{-}+1)}},
\end{eqnarray}
where we set $\hat{p}_{\varepsilon;\sigma}:=\max\{p_{\varepsilon;v}(x_{\sigma}),p_{\varepsilon;v}(y_{\sigma})\}$ and $\ti{p}_{\varepsilon;\sigma}:=\min\{p_{\varepsilon;v}(x_{\sigma}),p_{\varepsilon;v}(y_{\sigma})\}$. Via \eqref{assf}, \eqref{f}, \eqref{80} and using that $a_{\varepsilon;v}(\cdot),b_{\varepsilon;v}(\cdot)\ge \varepsilon$ and $\ell_{\varepsilon}(\sigma^{-1}(x_{\sigma}-y_{\sigma}))\ge \varepsilon$, we manipulate \eqref{66} to obtain
\begin{eqnarray*}
\frac{2\nr{f}_{L^{\infty}(\Omega)}\snr{p_{\varepsilon;v}(x_{\sigma})-p_{\varepsilon;v}(y_{\sigma})}}{\varepsilon^{6(p^{+}+p_{-}+1)}}&+&\frac{\nr{f}_{L^{\infty}(\Omega)}\snr{a_{\varepsilon;v}(x_{\sigma})-a_{\varepsilon;v}(y_{\sigma})}}{\varepsilon^{2q}}\nonumber \\
&+&\frac{\nr{f}_{L^{\infty}(\Omega)}\snr{b_{\varepsilon;v}(x_{\sigma})-b_{\varepsilon;v}(y_{\sigma})}}{\varepsilon^{2s}}+\frac{\snr{f(x_{\sigma})-f(y_{\sigma})}}{\varepsilon^{2(p^{+}+p_{-}+1)}}\nonumber \\
&\ge&\nr{f}_{L^{\infty}(\Omega)}\mathcal{L}(\varepsilon,\sigma)+\frac{\nr{f}_{L^{\infty}(\Omega)}\snr{a_{\varepsilon;v}(x_{\sigma})-a_{\varepsilon;v}(y_{\sigma})}}{a_{\varepsilon;v}(x_{\sigma})a_{\varepsilon;v}(y_{\sigma})\ell_{\varepsilon}\left(\frac{x_{\sigma}-y_{\sigma}}{\sigma}\right)^{q}}\nonumber \\
&+&\frac{\nr{f}_{L^{\infty}(\Omega)}\snr{b_{\varepsilon;v}(x_{\sigma})-b_{\varepsilon;v}(y_{\sigma})}}{b_{\varepsilon;v}(x_{\sigma})b_{\varepsilon;v}(y_{\sigma})\ell_{\varepsilon}\left(\frac{x_{\sigma}-y_{\sigma}}{\sigma}\right)^{s}}+\frac{\snr{f(x_{\sigma})-f(y_{\sigma})}}{G_{\varepsilon}\left(y_{\sigma},\frac{x_{\sigma}-y_{\sigma}}{\sigma}\right)}\nonumber \\
&\ge&\frac{f(x_{\sigma})}{G_{\varepsilon}\left(x_{\sigma},\frac{x_{\sigma}-y_{\sigma}}{\sigma}\right)}-\frac{f(y_{\sigma})}{G_{\varepsilon}\left(y_{\sigma},\frac{x_{\sigma}-y_{\sigma}}{\sigma}\right)}\nonumber \\
&\ge&\varepsilon\omega_{0},
\end{eqnarray*}
therefore it is
\begin{flalign}\label{4}
\frac{2\nr{f}_{L^{\infty}(\Omega)}\snr{p_{\varepsilon;v}(x_{\sigma})-p_{\varepsilon;v}(y_{\sigma})}}{\varepsilon^{6(p^{+}+p_{-}+1)+1}}&+\frac{\snr{f(x_{\sigma})-f(y_{\sigma})}}{\varepsilon^{2(p^{+}+p_{-}+1)}}\nonumber \\
&+\nr{f}_{L^{\infty}(\Omega)}\left[\frac{\snr{a_{\varepsilon;v}(x_{\sigma})-a_{\varepsilon;v}(y_{\sigma})}}{\varepsilon^{2q+1}}+\frac{\snr{b_{\varepsilon;v}(x_{\sigma})-b_{\varepsilon;v}(y_{\sigma})}}{\varepsilon^{2s+1}}\right]\ge \omega_{0}.
\end{flalign}
Recalling the $f(\cdot)$, $p_{\varepsilon;v}(\cdot)$ $a_{\varepsilon;v}(\cdot)$ and $b_{\varepsilon;v}(\cdot)$ are continuous and that $\snr{x_{\sigma}-y_{\sigma}}\to 0$ by \eqref{0}, we can send $\sigma\to 0$ in \eqref{4} to reach a contradiction with \eqref{1}. The proof is complete.
\end{proof}
At this stage, we need to construct continuous viscosity subsolutions and supersolutions of \eqref{eqe} with a fixed boundary datum. 
\begin{lemma}\label{subsux}
Let $\Omega\subset \mathbb{R}^{n}$ be an open, bounded domain satisfying a uniform exterior sphere condition. Assume \texttt{set} and let $g\in C(\partial \Omega)$ be any function with modulus of continuity $\omega_{g}(\cdot)$. Then equation \eqref{eqe} admits a viscosity supersolution $\overline{w}\in C(\bar{\Omega})$ and a viscosity subsolution $\underline{w}\in C(\bar{\Omega})$ for all numbers $\varepsilon\in (0,1)$, maps $v\in C(\bar{\Omega})$ so that $\left.\underline{w}\right|_{\partial\Omega}=\left.\overline{w}\right|_{\partial \Omega}=g$.
\end{lemma}
\begin{proof}
The proof closely follows that of \cite[Lemma 2]{hprs}, see also \cite[Proposition 3.2]{ckls}. We construct a continuous viscosity supersolution $\overline{w}$ to \eqref{eqe} agreeing with $g$ on $\partial \Omega$ for any $\varepsilon\in (0,1)$ and all function $v\in C(\bar{\Omega})$. The construction of a subsolution $\underline{w}$ with analogous features can be obtained in a similar way. Let $x_{0}\in \mathbb{R}^{n}$ be any point with $\dist(x_{0},\partial\Omega)\ge 1$, set
$$
\Gamma_{1}:=\max\{\nr{f}_{L^{\infty}(\Omega)},\lambda n\},\qquad \Gamma_{2}:=\frac{16(1+\diam(\Omega))^{2}\Gamma_{1}}{\lambda n}+\nr{g}_{L^{\infty}(\partial \Omega)}
$$
and define the function $$\ti{w}(x):=\Gamma_{2}-\Gamma_{1}\snr{x-x_{0}}^{2}(2\lambda n)^{-1}.$$
The choice of $\Gamma_{1}$ and $\Gamma_{2}$ yields that $\snr{D\ti{w}}\ge 1$, $\ti{w}\ge 0$ in $\Omega$ and that $\left.\ti{w}\right|_{\partial \Omega}\ge \nr{g}_{L^{\infty}(\partial \Omega)}$. Now, notice that $F(D^{2}\ti{w})\ge 0$, in fact, being the identity positive definite it is
\begin{eqnarray*}
F(D^{2}\ti{w})\stackrel{\eqref{assf}_{2}}{=}F\left(-\frac{\Gamma_{1}\mathbf{I}}{\lambda n}\right)-F(0)\stackrel{\eqref{ell}}{\ge}\lambda \texttt{tr}\left(\frac{\Gamma_{1}\mathbf{I}}{\lambda n}\right)\ge \Gamma_{1},
\end{eqnarray*}
so for all $x\in \Omega$ it is
\begin{eqnarray*}
\left[\ell_{\varepsilon}(D\ti{w})^{p_{\varepsilon;v}(x)}+a_{\varepsilon;v}(x)\ell_{\varepsilon}(D\ti{w})^{q}+b_{\varepsilon;v}(x)\ell_{\varepsilon}(D\ti{w})^{s}\right]\left(\varepsilon \ti{w}+F(D^{2}\ti{w})\right)\ge F(D^{2}\ti{w})\ge f(x),
\end{eqnarray*}
because of the very definition of $\Gamma_{1}$. Let $r_{*}\equiv r_{*}(\partial \Omega)>0$ be the radius provided by the uniform exterior sphere condition, $y\in \partial \Omega$ be any point and $x_{y}\in \mathbb{R}^{n}$ be so that $\snr{y-x_{y}}=r_{*}$ and $\bar{B}_{r_{*}}(x_{y})\cap \bar{\Omega}=\{y\}$. Let $\ti{r}:=r_{*}+\diam(\Omega)$, $\gamma>\max\left\{2,\frac{1}{\lambda}+n\frac{\Lambda}{\lambda}\right\}$, $L>0$ and $$w_{y}(x):=L(r_{*}^{-\gamma}-\snr{x-x_{y}}^{-\gamma}).$$ By construction it is $w_{y}(y)=0$, $w_{y}(x)\ge 0$ for $x\in \Omega$ and
\begin{eqnarray*}
Dw_{y}(x):=L\gamma\frac{x-x_{y}}{\snr{x-x_{y}}^{\gamma+2}},\qquad D^{2}w_{y}=\frac{L\gamma}{\snr{x-x_{y}}^{\gamma+2}}\left[\mathbf{I}-(\gamma+2)\frac{(x-x_{y})\otimes (x-x_{y})}{\snr{x-x_{y}}^{2}}\right],
\end{eqnarray*}
so we can control from below
\begin{eqnarray}\label{7.1}
\snr{Dw_{y}}\ge L\gamma\ti{r}^{-(\gamma+1)}\quad \mbox{in} \ \ \Omega
\end{eqnarray}
and
\begin{eqnarray*}
F(D^{2}w_{y})&\stackrel{\eqref{assf}_{2}}{=}&\left(F(D^{2}w_{y})-F\left(\frac{L\gamma\mathbf{I}}{\snr{x-x_{y}}^{\gamma+2}}\right)\right)+\left(F\left(\frac{L\gamma\mathbf{I}}{\snr{x-x_{y}}^{\gamma+2}}\right)-F(0)\right)\nonumber \\
&\stackrel{\eqref{ell}}{\ge}&\frac{L\gamma\lambda(\gamma+2)}{\snr{x-x_{y}}^{\gamma+4}}\texttt{tr}\left((x-x_{y})\otimes (x-x_{y})\right)-\frac{L\Lambda\gamma}{\snr{x-x_{y}}^{\gamma+2}}\texttt{tr}(\mathbf{I})\nonumber \\
&=&\frac{L\gamma(\lambda(\gamma+2)-n\Lambda)}{\snr{x-x_{y}}^{\gamma+2}}\ge \frac{L\gamma}{\snr{x-x_{y}}^{\gamma+2}},
\end{eqnarray*}
where also used the lower bound imposed on $\gamma$. We stress that the restrictions imposed on the size of $\gamma$ yield that $\gamma\equiv \gamma(n,\lambda,\Lambda)$. At this stage, we select $L>0$ so that
\begin{eqnarray}\label{7}
\frac{L\gamma}{\ti{r}^{\gamma+1}}\ge 1\qquad \mbox{and}\qquad \frac{L\gamma}{\ti{r}^{2+\gamma}}\ge \nr{f}_{L^{\infty}(\Omega)}+\nr{g}_{L^{\infty}(\partial\Omega)},
\end{eqnarray}
thus fixing the dependency $L\equiv L(n,\lambda,\lambda,\nr{f}_{L^{\infty}(\Omega)},\nr{g}_{L^{\infty}(\partial \Omega)},\partial \Omega,\diam(\Omega))$. Now, let $\tau\in (0,1)$ be any number and define the function
\begin{eqnarray*}
w_{y;\tau}(x):=g(y)+\tau+\Gamma_{\tau}w_{y}(x),
\end{eqnarray*}
where $\Gamma_{\tau}\ge1$ is selected in such a way that $w_{y;\tau}(x)\ge g(x)$ for $x\in \partial \Omega$. This can be done by defining
\begin{eqnarray*}
\Gamma_{\tau}:=4\left(\sup_{x\in \partial \Omega, x\not =y}\frac{\left(\omega_{g}(\snr{x-y})-\tau\right)_{+}}{w_{y}(x)}\right)+1.
\end{eqnarray*}
The uniform sphere condition imposed on $\partial \Omega$ yields that $\Gamma_{\tau}$ does not depend on $y\in \partial \Omega$. We then estimate using the very definition of $w_{y;\tau}(\cdot)$, \eqref{7.1} and \eqref{7},
\begin{flalign*}
&\left[\ell_{\varepsilon}(Dw_{y;\tau})^{p_{\varepsilon;v}(x)}+a_{\varepsilon;v}(x)\ell_{\varepsilon}(Dw_{y;\tau})^{q}+b_{\varepsilon;v}(x)\ell_{\varepsilon}(Dw_{y;\tau})^{s}\right]\left(\varepsilon w_{y;\tau}+F(D^{2}w_{y;\tau})\right)\nonumber \\
&\qquad \qquad \qquad \ge -\nr{g}_{L^{\infty}(\partial \Omega)}+\Gamma_{\tau}\left(\nr{f}_{L^{\infty}(\Omega)}+\nr{g}_{L^{\infty}(\partial \Omega)}\right)\ge \nr{f}_{L^{\infty}(\Omega)}\ge f(x),
\end{flalign*}
which means that $w_{y;\tau}$ is a viscosity supersolution of equation \eqref{eqe} for all $y\in \partial \Omega$ and all $\tau\in (0,1)$, and, as a consequence, the map $\ti{w}_{y;\tau}:=\min\left\{\ti{w},w_{y;\tau}\right\}$ is a viscosity supersolution of \eqref{eqe}. Finally, setting
\begin{eqnarray*}
\overline{w}(x):=\inf\left\{\ti{w}_{y;\tau}(x)\colon y\in \partial \Omega, \ \tau\in (0,1)\right\}
\end{eqnarray*}
we obtain the required viscosity supersolution to \eqref{eqe} agreeing with $g(\cdot)$ on $\partial \Omega$.
\end{proof}
As a consequence of the two above lemmas, we obtain the existence of a continuous viscosity solution to equation \eqref{eqe}.
\begin{corollary}\label{c1}
Let $\Omega\subset \mathbb{R}^{n}$ be an open, bounded domain satisfying the uniform sphere condition and assume \texttt{set}. Then, for any $g\in C(\partial \Omega)$ and $v\in C(\bar{\Omega})$ there exists a viscosity solution $u_{\varepsilon}\in C(\bar{\Omega})$ to equation \eqref{eqe} so that $\underline{w}\le u_{\varepsilon}\le \overline{w}$, where $\underline{w}$, $\overline{w}$ are respectively the subsolution and the supersolution constructed in Lemma \ref{subsux}. In particular, whenever $\Omega'\Subset \Omega$ is an open set it holds that
\begin{flalign}\label{8}
\nr{u_{\varepsilon}}_{L^{\infty}(\Omega')}\le c(n,\lambda,\Lambda,\nr{f}_{L^{\infty}(\Omega)},\nr{g}_{L^{\infty}(\partial \Omega)},\partial \Omega,\diam(\Omega),\dist(\Omega',\partial \Omega)).
\end{flalign}
\end{corollary}
\begin{proof}
The proof immediately follows by combining \cite[Theorem 4.1]{cil} with Lemmas \ref{comp}-\ref{subsux}. 
\end{proof}
Now we are ready to show the existence of a viscosity solution of Dirichlet problem
\begin{flalign}\label{pdv}
\begin{cases}
\ \left[\snr{Du_{v}}^{p_{v}(x)}+a(x)\mathds{1}_{\{v>0\}}\snr{Du_{v}}^{q}+b(x)\mathds{1}_{\{v<0\}}\snr{Du_{v}}^{s}\right]F(D^{2}u_{v})=f(x)\quad &\mbox{in} \ \ \Omega\\
\ u_{v}=g\quad &\mbox{on} \ \ \partial \Omega,
\end{cases}
\end{flalign}
where $v\in C(\bar{\Omega})$, $g\in C(\partial \Omega)$ and assumptions \texttt{set} are in force. 
\begin{corollary}\label{c2}
Let $\Omega\subset \mathbb{R}^{n}$ be an open, bounded domain satisfying the uniform sphere condition and assume \texttt{set}. Then, for any $g\in C(\partial \Omega)$ and $v\in C(\bar{\Omega})$, Dirichlet problem \eqref{pdv} admits a viscosity solution $u_{v}\in C(\bar{\Omega})$ so that $\left. u_{v}\right|_{\partial \Omega}=\left. g\right|_{\partial \Omega}$ and $\underline{w}\le u_{v}\le \overline{w}$, where $\underline{w}$. $\overline{w}$ are respectively the subsolution and the supersolution constructed in Lemma \ref{subsux}. In particular, whenever $\Omega'\Subset \Omega$ is an open set it holds that
\begin{flalign}\label{10}
\nr{u_{v}}_{L^{\infty}(\Omega')}\le c(n,\lambda,\Lambda,\nr{f}_{L^{\infty}(\Omega)},\nr{g}_{L^{\infty}(\partial \Omega)},\partial \Omega,\diam(\Omega),\dist(\Omega',\partial \Omega))
\end{flalign}
and, for all $\beta_{0}\in (0,1)$ it is
\begin{eqnarray}\label{12}
[u]_{0,\beta_{0};\Omega'}\le c(n,\lambda,\Lambda,\nr{f}_{L^{\infty}(\Omega)},\nr{g}_{L^{\infty}(\partial \Omega)},\partial \Omega,\diam(\Omega),\dist(\Omega',\partial \Omega),\beta_{0}).
\end{eqnarray}
\end{corollary}
\begin{proof}
By Corollary \ref{c1}, there exists a viscosity solution $u_{\varepsilon}\in C(\bar{\Omega})$ to equation \eqref{eqe} so that
\begin{eqnarray}\label{9}
\left. u_{\varepsilon}\right|_{\partial \Omega}=\left. g\right|_{\partial \Omega}\qquad \mbox{and}\qquad \underline{w}\le u_{\varepsilon}\le \overline{w},
\end{eqnarray}
where $\underline{w}$, $\overline{w}$ are respectively the viscosity subsolution and the viscosity supersolution to \eqref{eqe} determined by Lemma \ref{subsux}. We stress that $\underline{w}$ and $\overline{w}$ do not depend from $\varepsilon$. Notice that the bound in \eqref{8} is uniform in $\varepsilon$ and that equation \eqref{eqe} falls in the class of those considered by Proposition \ref{prophol}, with $\mu=\varepsilon$, $p(\cdot)\equiv p_{\varepsilon;v}(\cdot)$, $q(\cdot)\equiv q$, $s(\cdot)\equiv s$, $a(\cdot)\equiv a_{\varepsilon;v}(\cdot)$ and $b(\cdot)\equiv b_{\varepsilon;v}(\cdot)$, therefore, keeping in mind Remark \ref{r1}, we see that $\{u_{\varepsilon}\}\subset C^{0,\beta_{0}}_{\loc}(\Omega)$ for all $\beta_{0}\in (0,1)$ with uniform estimates on the H\"older seminorm, cf. \eqref{a.es}. This, together with \eqref{9}, the compact embedding of the H\"older spaces $C^{0,\beta_{1}}(\Omega')\hookrightarrow C^{0,\beta_{2}}(\Omega')$ for $\beta_{2}<\beta_{1}$ and \eqref{8}, gives that $u_{\varepsilon}\to u_{v}$ uniformly on compact subsets of $\Omega$, so we have
\begin{flalign}\label{11}
u_{v}\in C(\bar{\Omega}),\qquad \left. u_{v}\right|_{\partial \Omega}=\left. g\right|_{\partial \Omega},\qquad \underline{w}\le u_{v}\le \overline{w},\qquad \nr{u_{v}}_{L^{\infty}(\Omega')}+[u_{v}]_{0,\beta_{0};\Omega'}\le c,
\end{flalign}
with $c\equiv c(n,\lambda,\Lambda,\nr{f}_{L^{\infty}(\Omega)},\nr{g}_{L^{\infty}(\partial \Omega)},\partial \Omega,\diam(\Omega),\dist(\Omega',\partial \Omega),\beta_{0})$ for all $\beta_{0}\in (0,1)$ (of course the dependency from $\beta_{0}$ occurs only when considering $[u_{v}]_{0,\beta_{0};\Omega'}$). Finally, by very definition, we have that $p_{\varepsilon;v}\to p_{v}$, $a_{\varepsilon;v}\to a\mathds{1}_{\{v>0\}}$ and $b_{\varepsilon;v}\to b\mathds{1}_{\{v<0\}}$ in $\Omega$, so by well-known stability properties of viscosity solutions, cf. \cite[Chapter 3]{ka} and $\eqref{11}_{2}$ we have that $u_{v}\in C(\bar{\Omega})$ is a viscosity solution of equation \eqref{pdv}. 
\end{proof}
\subsection{Proof of Theorem \ref{t1}}
Let $\ti{u}\in C(\bar{\Omega})$ be any function. We recursively define the sequence of functions $\{u_{\kk}\}_{\kk\in \N\cup\{0\}}$ so that $u_{0}=\ti{u}$ and for $\kk\ge 1$, $u_{\kk}$ is a solution of problem \eqref{pdv} with $v\equiv u_{\kk-1}$, whose existence is assured by Corollary \ref{c2}. As the bounds in \eqref{10}-\eqref{12} do not depend on $v$ and so in our case they are independent from $\kk$, we have that sequence $\{u_{\kk}\}$ is uniformly bounded with respect to the full $C^{0,\beta_{0}}$-norm for all $\beta_{0}\in (0,1)$, therefore $u_{\kk}\to u_{\infty}$ uniformly on compact subsets of $\Omega$, $u_{\infty}\in C(\bar{\Omega})$ and $\left.u_{\infty}\right|_{\partial \Omega}=\left.g\right|_{\partial \Omega}$. Standard stability results, see \cite[Chapter 3]{ka} eventually render that $u_{\infty}$ is a viscosity solution of problem \eqref{pd} and the proof is complete. 

\section{Compactness for switched differential inequalities}\label{hr}
The main result of this section is uniform H\"older continuity for viscosity solutions of the switched equation \eqref{d.1}. The uniformity is due to the fact that all the constants bounding the H\"older seminorm of solutions will not depend from $\xi$, or from the moduli of continuity of coefficients $a(\cdot)$, $b(\cdot)$ nor on their $L^{\infty}$-norm.

\begin{proposition}\label{phol}
Under assumptions \texttt{set}, let $u\in C(\Omega)$ be a viscosity subsolution of \eqref{dif1} and a viscosity supersolution to \eqref{dif2}. Then $u\in C^{0,\beta_{0}}_{\loc}(\Omega)$ for all $\beta_{0}\in (0,1)$. In particular, for any $\beta_{0}\in (0,1)$ there exists a threshold radius $r_{*}\equiv r_{*}(\beta_{0})\in (0,1/4)$ so that whenever $B_{\rr}(z_{0})\Subset \Omega$ is a ball with $\rr\in (0,r_{*}]$ it holds that
\begin{eqnarray*}
\snr{u(x)-u(y)}\le c\snr{x-y}^{\beta_{0}}\qquad \mbox{for all} \ \ x,y\in B_{\rr/2}(z_{0}),
\end{eqnarray*}
with $c\equiv c(n,\lambda,\Lambda,\nr{u}_{L^{\infty}(B_{\rr}(z_{0}))},\nr{f}_{L^{\infty}(\Omega)},\rr,\beta_{0})$.
\end{proposition}

\begin{proof}
Let $u\in C(\Omega)$ be a viscosity subsolution to \eqref{dif1} and a viscosity supersolution of \eqref{dif2}, $\beta_{0}\in (0,1)$ be any number and $B_{\rr}(z_{0})\Subset \Omega$ be a ball with radius $\rr\in \left(0,r_{*}\right]$, where $r_{*}:=\left(\beta_{0}/10\right)^{\frac{1}{1-\beta_{0}}}$ is a threshold radius that will play an important role in a few lines. We aim to show that there are two positive constants, $A_{1}\equiv A_{1}(n,\lambda,\Lambda,\nr{u}_{L^{\infty}(\Omega)},\nr{f}_{L^{\infty}(\Omega)},\rr,\beta_{0})$, $A_{2}\equiv A_{2}(\nr{u}_{L^{\infty}(\Omega)},\rr)$ so that
\begin{flalign}\label{12.1}
\mathcal{M}(x_{0}):=\sup_{x,y\in B_{\rr}(z_{0})}\left(u(x)-u(y)-A_{1}\omega(\snr{x-y})-A_{2}\left(\snr{x-x_{0}}^{2}+\snr{y-x_{0}}^{2}\right)\right)\le 0
\end{flalign}
holds for all $x_{0}\in B_{\rr/2}(z_{0})$. In \eqref{12.1}, it is
\begin{flalign*}
\omega(t):=t^{\beta_{0}} \ \ \mbox{if} \ \ \snr{\xi}\le \kk_{0}^{-1},\qquad \omega(t):=\begin{cases}\ t-\omega_{0}t^{3/2}\ \ &\mbox{if} \ \ t\le t_{0}\\
\ \omega(t_{0}) \ \ &\mbox{if} \ \ t>t_{0} 
\end{cases}\ \ \mbox{if} \ \ \snr{\xi}>\kk_{0}^{-1},
\end{flalign*}
where $\kk_{0}:=(2(A_{1}+2A_{2}))^{-1}$ is a limiting number, $\omega_{0}=1/3$ and $t_{0}:=\left(2/(3\omega_{0})\right)^{2}\ge 1$. By contradiction, we assume that
\begin{eqnarray}\label{13}
\textnormal{there exists}\ x_{0}\in B_{\rr/2}(z_{0}) \ \textnormal{such that} \ \mathcal{M}(x_{0})> 0 \ \textnormal{for all positive} \ A_{1},A_{2},
\end{eqnarray}
define quantities
\begin{flalign}\label{14}
\left\{
\begin{array}{c}
\displaystyle 
\ A_{1}:=\frac{4}{\beta_{0}(1-\beta_{0})}\left[\frac{\nr{f}_{L^{\infty}(\Omega)}}{\lambda}+(2A_{2}+1)\left(\frac{\Lambda}{\lambda}(n-1)+1\right)\right]\\[17pt] \displaystyle
\ A_{2}:=64\rr^{-2}\nr{u}_{L^{\infty}(\Omega)}
\end{array}
\right.
\end{flalign}
and consider the auxiliary functions
\begin{eqnarray*}
\begin{cases}
\ \psi(x,y):=A_{1}\omega(\snr{x-y})+A_{2}\left(\snr{x-x_{0}}^{2}+\snr{y-x_{0}}^{2}\right)\\
\ \phi(x,y):=u(x)-u(y)-\psi(x,y).
\end{cases}
\end{eqnarray*}
If $(\bar{x},\bar{y})\in \bar{B}_{\rr}(z_{0})\times \bar{B}_{\rr}(z_{0})$ is a maximum point of $\phi(\cdot)$, via \eqref{13} we have $\phi(\bar{x},\bar{y})=\mathcal{M}(x_{0})>0$, so
\begin{eqnarray*}
A_{1}\omega(\snr{\bar{x}-\bar{y}})+A_{2}\left(\snr{\bar{x}-x_{0}}^{2}+\snr{\bar{y}-x_{0}}^{2}\right)\le u(\bar{x})-u(\bar{y})\le 2\nr{u}_{L^{\infty}(\Omega)}.
\end{eqnarray*}
Inserting $\eqref{14}_{2}$ in the above inequality yields that $\bar{x}$, $\bar{y}$ both belong to the interior of $B_{\rr}(z_{0})$:
\begin{eqnarray*}
\snr{\bar{x}-z_{0}}\le \snr{\bar{x}-x_{0}}+\snr{x_{0}-z_{0}}\le \frac{3\rr}{4}\qquad \mbox{and}\qquad \snr{\bar{y}-z_{0}}\le \snr{\bar{y}-x_{0}}+\snr{x_{0}-z_{0}}\le \frac{3\rr}{4}.
\end{eqnarray*}
Moreover, $\bar{x}\not =\bar{y}$, otherwise $\mathcal{M}(x_{0})=\phi(\bar{x},\bar{y})=0$ and \eqref{12.1} would be verified. This last remark implies that $\psi(\cdot)$ is smooth in a small neighborhood of $(\bar{x},\bar{y})$, therefore we can determine its gradients
\begin{flalign*}
&\xi_{\bar{x}}:=\partial_{x}\psi(\bar{x},\bar{y})=A_{1}\omega'(\snr{\bar{x}-\bar{y}})\frac{\bar{x}-\bar{y}}{\snr{\bar{x}-\bar{y}}}+2A_{2}(\bar{x}-x_{0}),\\
&\xi_{\bar{y}}:=-\partial_{y}\psi(\bar{x},\bar{y})=A_{1}\omega'(\snr{\bar{x}-\bar{y}})\frac{\bar{x}-\bar{y}}{\snr{\bar{x}-\bar{y}}}-2A_{2}(\bar{y}-x_{0}).
\end{flalign*}
To summarize, we have that $\phi(\cdot)$ attains its maximum in $(\bar{x},\bar{y})$ inside $B_{\rr}(z_{0})\times B_{\rr}(z_{0})$ and $\phi(\cdot)$ is smooth around $(\bar{x},\bar{y})$, thus Proposition \ref{p1} applies: for any $\iota>0$ we can find a threshold $\hat{\delta}=\hat{\delta}(\iota,\nr{D^{2}\psi})$ such that for all $\delta \in (0,\hat{\delta})$ the couple $(\xi_{\bar{x}},X_{\delta})$ is a limiting subjet of $u$ at $\bar{x}$ and the couple $(\xi_{\bar{y}},Y_{\delta})$ is a limiting superjet of $u$ at $\bar{y}$ and the matrix inequality
\begin{flalign}\label{15}
\begin{bmatrix}
X_{\delta} & 0 \\ 0 & -Y_{\delta} 
\end{bmatrix}\le \begin{bmatrix}
Z & -Z \\ -Z & Z 
\end{bmatrix}+(2A_{2}+\delta)\mathbf{I}
\end{flalign}
holds, where we set
\begin{flalign*}
Z:=&A_{1}(D^{2}\omega)(\snr{\bar{x}-\bar{y}})\nonumber \\
=&A_{1}\left[\frac{\omega'(\snr{\bar{x}-\bar{y}})}{\snr{\bar{x}-\bar{y}}}\mathbf{I}+\left(\omega''(\snr{\bar{x}-\bar{y}})-\frac{\omega'(\snr{\bar{x}-\bar{y}})}{\snr{\bar{x}-\bar{y}}}\right)\frac{(\bar{x}-\bar{y})\otimes (\bar{x}-\bar{y})}{\snr{\bar{x}-\bar{y}}^{2}}\right].
\end{flalign*}
We fix $\delta\equiv \min\left\{1,\frac{\hat{\delta}}{4}\right\}$ and apply \eqref{15} to vectors of the form $(z,z)\in \mathbb{R}^{2n}$, to obtain  
\begin{flalign*}
\langle(X_{\delta}-Y_{\delta})z,z \rangle\le (4A_{2}+2)\snr{z}^{2}.
\end{flalign*}
This means that 
\begin{flalign}\label{16}
\mbox{all the eigenvalues of} \ \ X_{\delta}-Y_{\delta} \ \ \mbox{are less than or equal to} \ \ 2(2A_{2}+1).
\end{flalign}
In particular, applying \eqref{15} to the vector $\bar{z}:=\left(\frac{\bar{x}-\bar{y}}{\snr{\bar{x}-\bar{y}}},\frac{\bar{y}-\bar{x}}{\snr{\bar{x}-\bar{y}}}\right)$, we get
\begin{flalign*}
\left \langle (X_{\delta}-Y_{\delta})\frac{\bar{x}-\bar{y}}{\snr{\bar{x}-\bar{y}}},\right.&\left.\frac{\bar{x}-\bar{y}}{\snr{\bar{x}-\bar{y}}} \right \rangle\le 2(2A_{2}+1)+4A_{1}\omega''(\snr{\bar{x}-\bar{y}}).
\end{flalign*}
This yields in particular that
\begin{flalign}\label{17}
\mbox{at least one eigenvalue of} \ \ X_{\delta}-Y_{\delta} \ \ \mbox{is less than} \ \ 2(2A_{2}+1)+4A_{1}\omega''(\snr{\bar{x}-\bar{y}}).
\end{flalign}
As by definition $\omega''(t)<0$, we can majorize the quantity appearing in \eqref{17} as 
\begin{eqnarray*}
2(2A_{2}+1)+4A_{1}\omega''(\snr{\bar{x}-\bar{y}})\le 2(2A_{2}+1)-4A_{1}\snr{\omega''(1)}\stackrel{\eqref{14}_{1}}{<}0,
\end{eqnarray*}
where we also used that $\snr{\bar{x}-\bar{y}}\le 1/2$. This means that at least one eigenvalue of $X_{\delta}-Y_{\delta}$ is negative, thus via \eqref{m-}$_{2}$, \eqref{16} and \eqref{17}, we obtain
\begin{eqnarray}\label{18}
\mathcal{M}_{\lambda,\Lambda}^{-}(X_{\delta}-Y_{\delta})\ge -2(2A_{2}+1)\left[\Lambda(n-1)+\lambda\right]+4\lambda A_{1}\snr{\omega''(1)},
\end{eqnarray}
therefore
\begin{flalign}
F(X_{\delta})-F(Y_{\delta})\stackrel{\eqref{elll}}{\ge}\mathcal{M}^{-}_{\lambda,\Lambda}(X_{\delta}-Y_{\delta})\stackrel{\eqref{18}}{\ge}-2(2A_{2}+1)\left[\Lambda(n-1)+\lambda\right]+4\lambda A_{1}\snr{\omega''(1)}.\label{20}
\end{flalign}
With $\xi_{\bar{x}},\xi_{\bar{y}}$ computed before, we write the viscosity inequalities deriving from \eqref{dif1}-\eqref{dif2}:
\begin{flalign}\label{19}
\begin{cases}
\ \min\left\{F(X_{\delta}),H_{q}(\bar{x},\xi_{\bar{x}};\xi)F(X_{\delta}),H_{s}(\bar{x},\xi_{\bar{x}};\xi)F(X_{\delta})\right\}\le \nr{f}_{L^{\infty}(\Omega)}\\
\ \max\left\{F(Y_{\delta}),H_{q}(\bar{y},\xi_{\bar{y}};\xi)F(Y_{\delta}),H_{s}(\bar{y},\xi_{\bar{y}};\xi)F(Y_{\delta})\right\}\ge- \nr{f}_{L^{\infty}(\Omega)}.
\end{cases}
\end{flalign}
For simplicity, define
\begin{flalign*}
&\mathfrak{H}^{-}(x):= \min\left\{F(X_{\delta}),H_{q}(x,\xi_{x};\xi)F(X_{\delta}),H_{s}(x,\xi_{x};\xi)F(X_{\delta})\right\}\\
&\mathfrak{H}^{+}(x):=\max\left\{F(Y_{\delta}),H_{q}(x,\xi_{x};\xi)F(Y_{\delta}),H_{s}(x,\xi_{x};\xi)F(Y_{\delta})\right\}
\end{flalign*}
and notice that it is
\begin{flalign}\label{24.1}
\begin{cases}
\ \mathfrak{H}^{-}(\bar{x})=\min\left\{1,H_{q}(\bar{x},\xi_{\bar{x}};\xi),H_{s}(\bar{x},\xi_{\bar{x}};\xi)\right\}F(X_{\delta})\quad &\mbox{if} \ \ F(X_{\delta})\ge 0\\
\ \mathfrak{H}^{-}(\bar{x})=\max\left\{1,H_{q}(\bar{x},\xi_{\bar{x}};\xi),H_{s}(\bar{x},\xi_{\bar{x}};\xi)\right\}F(X_{\delta})\quad &\mbox{if} \ \ F(X_{\delta})<0\\
\ \mathfrak{H}^{+}(\bar{y})=\max\left\{1,H_{q}(\bar{y},\xi_{\bar{y}};Y_{\delta}),H_{s}(\bar{y},\xi_{\bar{y}};\xi)\right\}F(Y_{\delta})\quad &\mbox{if} \ \ F(Y_{\delta})\ge 0\\
\ \mathfrak{H}^{+}(\bar{y})=\min\left\{1,H_{q}(\bar{y},\xi_{\bar{y}};\xi),H_{s}(\bar{y},\xi_{\bar{y}};\xi)\right\}F(Y_{\delta})\quad &\mbox{if} \ \ F(Y_{\delta})<0.
\end{cases}
\end{flalign}
At this stage, we treat separately two cases: $\snr{\xi}>\kk_{0}^{-1}$ and $\snr{\xi}\le \kk_{0}^{-1}$.
\subsubsection*{Case $\snr{\xi}>\kk_{0}^{-1}$} We expand the expression of $\omega(\cdot)$ in \eqref{20} to get
\begin{eqnarray}\label{21}
F(X_{\delta})-F(Y_{\delta})\ge -2(2A_{2}+1)[\Lambda(n-1)+\lambda]+\lambda A_{1}.
\end{eqnarray}
Moreover, our choice of $\kk_{0}$ assures that
\begin{eqnarray*}
\min\left\{\snr{\xi+\xi_{\bar{x}}},\snr{\xi+\xi_{\bar{y}}}\right\}\ge \kk_{0}^{-1}-\max\left\{\snr{\xi_{\bar{x}}},\snr{\xi_{\bar{y}}}\right\}\ge A_{1}+2A_{2}\ge 1,
\end{eqnarray*}
which implies
\begin{eqnarray}\label{23}
\mf{H}^{-}(\bar{x})\ge 1\qquad \mbox{and}\qquad \mf{H}^{+}(\bar{y})\ge 1.
\end{eqnarray}
Keeping in mind \eqref{24.1}, we can manipulate the variational inequalities \eqref{19} to get
\begin{eqnarray*}
4\nr{f}_{L^{\infty}(\Omega)}&\stackrel{\eqref{23},\eqref{24.1}}{\ge}&\frac{2\nr{f}_{L^{\infty}(\Omega)}}{\mf{H}^{-}(\bar{x})}+\frac{2\nr{f}_{L^{\infty}(\Omega)}}{\mf{H}^{+}(\bar{y})}\nonumber \\
&\stackrel{\eqref{19}}{\ge}&F(X_{\delta})-F(Y_{\delta})\nonumber \\
&\stackrel{\eqref{20}}{\ge}&-2(2A_{2}+1)\left[\Lambda(n-1)+\lambda\right]+\lambda A_{1},
\end{eqnarray*}
which renders that
\begin{eqnarray}\label{21.1}
4\nr{f}_{L^{\infty}(\Omega)}\ge -2(2A_{2}+1)\left[\Lambda(n-1)+\lambda\right]+\lambda A_{1}.
\end{eqnarray}
The content of \eqref{21.1} contradicts the choice made in \eqref{14}$_{1}$.
\subsubsection*{Case $\snr{\xi}\le \kk_{0}^{-1}$}
In this situation, \eqref{20} reads as
\begin{eqnarray}\label{25}
F(X_{\delta})-F(Y_{\delta})\ge -2(2A_{2}+1)[\Lambda(n-1)+\lambda]+4\beta_{0}(1-\beta_{0})\lambda A_{1},
\end{eqnarray}
Now set $\gamma_{*}:=\beta_{0}r_{*}^{\beta_{0}-1}$ and notice that our choice of $r_{*}$ yields that
$\sqrt{\gamma_{*}^{2}-4\gamma_{*}}-6>1$, so 
\begin{eqnarray*}
\min\left\{\snr{\xi_{\bar{x}}},\snr{\xi_{\bar{y}}}\right\}-\kk_{0}^{-1}&\ge&\sqrt{A_{1}\beta_{0}\snr{\bar{x}-\bar{y}}^{\beta_{0}-1}\left(A_{1}\beta_{0}\snr{\bar{x}-\bar{y}}^{\beta_{0}-1}-4A_{2}\right)}-\kk_{0}^{-1}\nonumber \\
&\ge&\sqrt{A_{1}\gamma_{*}\left(A_{1}\gamma_{*}-4A_{2}\right)}-2(A_{1}+2A_{2})\nonumber \\
&\stackrel{\eqref{14}}{\ge}&A_{1}\left(\sqrt{\gamma_{*}^{2}-4\gamma_{*}}-6\right)\ge A_{1}>1
\end{eqnarray*}
and \eqref{23} holds in this case as well. Therefore we can combine as before the variational inequalities \eqref{19} with \eqref{25} and \eqref{23} to deduce
\begin{eqnarray*}
4\nr{f}_{L^{\infty}(\Omega)}\ge -2(2A_{2}+1)[\Lambda(n-1)+\lambda]+4\beta_{0}(1-\beta_{0})\lambda A_{1},
\end{eqnarray*}
which is again a contradiction of $\eqref{14}_{1}$.\\\\
Merging the two previous cases we can conclude that if $u\in C(\Omega)$ is a viscosity solution to \eqref{d.1}, then $u$ is $\beta_{0}$-H\"older continuous on $B_{\rr/2}(z_{0})$ for all $\beta_{0}\in (0,1)$ and estimate
$$
[u]_{0,\beta_{0};B_{\rr/2}(z_{0})}\le c(n,\lambda,\Lambda,\nr{u}_{L^{\infty}(\Omega)},\nr{f}_{L^{\infty}(\Omega)},\rr,\beta_{0})
$$
holds true. The arbitrariety of $B_{\rr}(z_{0})\Subset \Omega$ allows using a standard covering argument to deduce that $u\in C^{0,\beta_{0}}_{\loc}(\Omega)$ for all $\beta_{0}\in (0,1)$ and the proof is complete.

\end{proof}

\section{Gradient H\"older continuity}\label{ghgh}
In this section we prove that viscosity solutions of equation \eqref{eq} are locally $C^{1,\alpha_{0}}$-regular for some $\alpha_{0}\equiv \alpha_{0}(n,\lambda,\Lambda,p^{+},p_{-})\in (0,1)$. 
To do so, we shall first prove that in a suitable smallness regime, a continuous viscosity solution  of the switched equation \eqref{d.1} is $L^{\infty}$-close to a solution of a homogeneous problem of type \eqref{probhom}. This closeness is assured by an "harmonic" approximation lemma, whose proof is based on \cite{de1,is} and that strongly relies on the smallness of certain quantities and on the compactness earned via Proposition \ref{phol}.
\subsection{Smallness regime}\label{sr}
We exploit the scaling properties of \eqref{dif1}-\eqref{dif2} for reducing the problem to a smallness regime. In other terms, if $\xi\in \mathbb{R}^{n}$ is an arbitrary vector and $u\in C(\Omega)$ is a viscosity subsolution/supersolution to \eqref{dif1}/\eqref{dif2}, we blow and scale $u$ in order to construct another map $\mathfrak{u}$, that is a viscosity subsolution of an equation having the same structure of \eqref{dif1}, a viscosity supersolution of an equation similar to \eqref{dif2} and such that, for a given $\varepsilon\in (0,1)$ it is $\osc_{B_{1}(0)}\uu\le 1$ and the right-hand side constant appearing in \eqref{dif1}-\eqref{dif2} can be controlled in modulus by $\varepsilon$. Under these conditions, $\uu$ is called "$\varepsilon$-normalized viscosity solution". Let us show this construction. Let $\varepsilon\in (0,1)$ be any number, $B_{\tau}(x_{0})\Subset \Omega$ be any ball with with $\tau\in \left(0,\frac{1}{16}\min\{\diam(\Omega),1\}\right)$ to be quantified later on and define $\mf{M}:=16\left(1+\nr{u}_{L^{\infty}(\Omega)}+\nr{f}_{L^{\infty}(\Omega)}+\nr{f}_{L^{\infty}(\Omega)}^{\frac{1}{p^{+}+1}}+\nr{f}_{L^{\infty}(\Omega)}^{\frac{1}{p_{-}+1}}\right)$. Now, if $u\in C(\Omega)$ is a viscosity solution to \eqref{d.1} on $B_{\tau}(x_{0})$, then a straightforward computation shows that the map $\uu(x):=u(x_{0}+\tau x)\mf{M}^{-1}$ is in particular a $\varepsilon$-normalized viscosity subsolution of
\begin{flalign}\label{d.1t}
\min\left\{\mf{F}(D^{2}\uu),\mf{H}_{q}(x,D\uu;\bar{\xi})\mf{F}(D^{2}\uu),\mf{H}_{s}(x,D\uu;\bar{\xi})\mf{F}(D^{2}\uu)\right\}=\mf{C}\quad \mbox{in} \ \ B_{1}(0)
\end{flalign}
and a $\varepsilon$-normalized viscosity supersolution to
\begin{flalign}\label{d.2t}
\max\left\{\mf{F}(D^{2}\uu),\mf{H}_{q}(x,D\uu;\bar{\xi})\mf{F}(D^{2}\uu),\mf{H}_{s}(x,D\uu;\bar{\xi})\mf{F}(D^{2}\uu)\right\}=-\mf{C}\quad \mbox{in} \ \ B_{1}(0),
\end{flalign}
where we set
\begin{flalign*}
\left\{
\begin{array}{c}
\displaystyle 
\bar{\xi}:=\frac{\tau}{\mf{M}}\xi,\qquad \mf{a}(x):=\left(\frac{\mf{M}}{\tau}\right)^{q-p^{+}}a(x_{0}+\tau x),\qquad \mf{b}(x):=\left(\frac{\mf{M}}{\tau}\right)^{s-p_{-}}b(x_{0}+\tau x)\\[17pt]\displaystyle
\mf{H}_{q}(x,z;\bar{\xi}):=\snr{\bar{\xi}+z}^{p^{+}}+\mf{a}(x)\snr{\bar{\xi}+z}^{q},\qquad \mf{H}_{s}(x,z;\bar{\xi}):=\snr{\bar{\xi}+z}^{p_{-}}+\mf{b}(x)\snr{\bar{\xi}+z}^{s}\\[17pt]\displaystyle
\mf{F}(M):=\frac{\tau^{2}}{\mf{M}}F\left(\frac{\mf{M}}{\tau^{2}}M\right),\qquad \mf{C}:=\max\left\{\frac{\tau^{p^{+}+2}}{\mf{M}^{p^{+}+1}},\frac{\tau^{p_{-}+2}}{\mf{M}^{p_{-}+1}},\frac{\tau^{2}}{\mf{M}}\right\}\nr{f}_{L^{\infty}(\Omega)},
\end{array}
\right.
\end{flalign*}
as by \eqref{pqs} and being $\mf{M}\ge 1$ and $\tau\le 1$ it is
\begin{eqnarray*}
\max\left\{\frac{\tau^{p^{+}+2}}{\mf{M}^{p^{+}+1}},\frac{\tau^{p_{-}+2}}{\mf{M}^{p_{-}+1}},\frac{\tau^{2}}{\mf{M}}\right\}=\frac{\tau^{2}}{\mf{M}}.
\end{eqnarray*}
A quick computation shows that if \eqref{assf} is in force, then $\mf{F}(\cdot)$ is $(\lambda,\Lambda)$-elliptic as well and, if $\varepsilon\in (0,1)$ is the number introduced above, we fix $\tau=\varepsilon^{\frac{1}{2}}$. Therefore, by construction it is \begin{flalign}\label{30}
\left\{
\begin{array}{c}
\displaystyle 
\nr{\mf{a}}_{L^{\infty}(B_{1}(0))}\le \left(\frac{\mf{M}}{\tau}\right)^{q-p^{+}}\nr{a}_{L^{\infty}(B_{\tau}(x_{0}))},\qquad \nr{\mf{b}}_{L^{\infty}(B_{1}(0))}\le \left(\frac{\mf{M}}{\tau}\right)^{s-p_{-}}\nr{b}_{L^{\infty}(B_{\tau}(x_{0}))}\\[17pt]\displaystyle
\nr{\uu}_{L^{\infty}(B_{1}(0))}\le 1,\qquad \osc_{B_{1}(0)}\uu\le 1,\qquad \mf{C}\le \varepsilon.
\end{array}
\right.
\end{flalign}
Finally, notice that there is no loss of generality in assuming that $\uu(0)=0$ since the function $(\uu-\uu(0))$ is still a $\varepsilon$-normalized viscosity subsolution/supersolution of \eqref{d.1t}/\eqref{d.2t} and verifies all the conditions listed above. This is the announced smallness regime. Clearly, for $\xi\equiv 0$ we find a $\varepsilon$-normalized viscosity solution of equation \eqref{eq}. We refer to \cite[Section 2.3]{hprs} for the case in which no coefficients appear.
\begin{remark}\label{rema}
\emph{Due to the strong nonhomogeneity of \eqref{d.1} and \eqref{dif1}-\eqref{dif2}, the scaling factor $\tau$ appears also in the definition of $\mf{a}(\cdot)$ and $\mf{b}(\cdot)$ and forces the (quite dangerous) bounds in $\eqref{30}_{1}$. Anyway, the $L^{\infty}$-norms of $\mf{a}(\cdot)$ and $\mf{b}(\cdot)$ will never influence the constants appearing in the forthcoming estimates and it will ultimately fixed as a function of $\texttt{data}$.}
\end{remark}
\subsection{Harmonic approximation}
In the next lemma we show that, in a suitable smallness regime, continuous viscosity solutions of \eqref{d.1} are close to solutions of the homogeneous problem \eqref{probhom}.
\begin{lemma}\label{har}
Assume \texttt{set} and let $\sigma\equiv \sigma(n,\lambda,\Lambda)\in (0,1)$ be as in \eqref{0331}. Then, there exists a positive $\varepsilon_{0}\equiv \varepsilon_{0}(\texttt{data})\in (0,1)$ such that if $\uu\in C(B_{1}(0))$ is a $\varepsilon_{0}$-normalized viscosity subsolution of \eqref{d.1t} and a $\varepsilon_{0}$-normalized viscosity supersolution to \eqref{d.2t}, it is possible to find $\xi_{\sigma}\in \mathbb{R}^{n}$ such that
\begin{eqnarray*}
\osc_{B_{\sigma}(0)}\left(\uu-\xi_{\sigma}\cdot x\right)< \frac{\sigma}{2}.
\end{eqnarray*}
\end{lemma}
\begin{proof}
By contradiction, we find sequences of fully nonlinear operators $\{\mf{F}_{\kk}(\cdot)\}$ that are uniformly $(\lambda,\Lambda)$-elliptic, of vectors $\{\bar{\xi}_{\kk}\}\subset \mathbb{R}^{n}$, of nonnegative functions $\{\mf{a}_{\kk}(\cdot)\},\{\mf{b}_{\kk}(\cdot)\}\subset C(B_{1}(0))$, of numbers $\{\mf{C}_{\kk}\}\subset [0,\infty)$ so that $\mf{C}_{\kk}\le \kk^{-1}$  and of maps $\{\uu_{\kk}\}\in C(\Omega)$ that are $\kk^{-1}$-normalized viscosity subsolution to
\begin{flalign}\label{kk1}
\min\left\{\mf{F}_{\kk}(D^{2}\uu_{\kk}),\mf{H}_{q;\kk}(x,D\uu_{\kk};\bar{\xi}_{\kk})\mf{F}_{\kk}(D^{2}\uu_{\kk}),\mf{H}_{s;\kk}(x,D\uu_{\kk};\bar{\xi}_{\kk})\mf{F}_{\kk}(D^{2}\uu_{\kk})\right\}=\mf{C}_{\kk}
\end{flalign}
in $B_{1}(0)$ and $\kk^{-1}$-normalized viscosity supersolution of
\begin{flalign}\label{kk2}
\max\left\{\mf{F}_{\kk}(D^{2}\uu_{\kk}),\mf{H}_{q;\kk}(x,D\uu_{\kk};\bar{\xi}_{\kk})\mf{F}_{\kk}(D^{2}\uu_{\kk}),\mf{H}_{s;\kk}(x,D\uu_{\kk};\bar{\xi}_{\kk})\mf{F}_{\kk}(D^{2}\uu_{\kk})\right\}=-\mf{C}_{\kk}
\end{flalign}
in $B_{1}(0)$ for all $\kk\in \N$, $\uu_{\kk}(0)=0$ and
\begin{flalign}\label{31}
\sup_{\kk\in \N}\nr{\uu_{\kk}}_{L^{\infty}(B_{1}(0))}\le 1,\qquad \sup_{\kk\in \N}\left(\osc_{B_{1}(0)} \uu_{\kk}\right)\le 1, \qquad \osc_{B_{\sigma}(0)}\left(\uu_{\kk}-\xi\cdot x\right)\ge\frac{\sigma}{2}\quad \mbox{for all} \ \ \xi\in \mathbb{R}^{n}.
\end{flalign}
In \eqref{kk1}-\eqref{kk2}, $\mf{H}_{q;\kk}(\cdot)$, $\mf{H}_{s;\kk}(\cdot)$ are defined as in Section \ref{sr}, with $\mf{a}_{\kk}(\cdot)$ and $\mf{b}_{\kk}(\cdot)$ replacing $\mf{a}(\cdot)$ and $\mf{b}(\cdot)$ respectively. As the sequence $\{\mf{F}_{\kk}(\cdot)\}$ is uniformly $(\lambda,\Lambda)$-elliptic, we have that
\begin{eqnarray}\label{37}
\mf{F}_{\kk}(\cdot)\to \mf{F}_{\infty}(\cdot)\ \  \mbox{for some} \ \ \mf{F}_{\infty}\in C(\mathcal{S}(n),\mathbb{R}) \ \  \mbox{uniformly}\ \ (\lambda,\Lambda)\mbox{-elliptic}.
\end{eqnarray}
Moreover, Proposition \ref{phol} applies to renormalized viscosity subsolutions/supersolutions of \eqref{kk1}/\eqref{kk2} as all the estimates made in its proof do not involve in a quantitative way the coefficients. This means that $\{\uu_{\kk}\}\subset C^{0,\beta_{0}}_{\loc}(B_{1}(0))$ for all $\beta_{0}\in (0,1)$, so, recalling also $\eqref{31}_{1,2}$ and Arzela-Ascoli theorem we have that
\begin{eqnarray}\label{32}
\uu_{\kk}\to \uu_{\infty}\quad \mbox{locally uniformly on} \ \ B_{1}(0)
\end{eqnarray}
and, by \eqref{31} and \eqref{32} it holds that $\uu_{\infty}\in C(B_{1}(0))$ with
\begin{flalign}\label{33}
 \nr{\uu_{\infty}}_{L^{\infty}(B_{1}(0))}\le 1\qquad \mbox{and}\qquad \osc_{B_{\sigma}(0)}\left(\uu_{\infty}-\xi\cdot x\right)\ge \frac{\sigma}{2} \ \ \mbox{for all} \ \ \xi\in \mathbb{R}^{n}.
\end{flalign}
We aim to prove that $\uu_{\infty}\in C(B_{1}(0))$ is a viscosity solution of
\begin{eqnarray}\label{34}
\mf{F}_{\infty}(D^{2}\uu_{\infty})=0\qquad \mbox{in} \ \ B_{1}(0).
\end{eqnarray}
Let us show that $\uu_{\infty}$ is a viscosity supersolution of \eqref{34}. Let $\varphi\in C^{2}(B_{1}(0))$ be so that $\uu_{\infty}-\varphi$ admits a local strict minimum at $x_{0}\in B_{1}(0)$. There is no loss of generality in assuming that $\varphi(\cdot)$ is a quadratic polynomial, i.e.: $$\varphi(x):=\frac{1}{2}A(x-x_{0})\cdot(x-x_{0})+b\cdot (x-x_{0})+\uu_{\infty}(x_{0}).$$ By \eqref{32} and standard perturbations arguments \cite[Lemma 5]{ka} we have that there exists a sequence of points $\{x_{\kk}\}\subset B_{1}(0)$ so that $x_{\kk}\to x_{0}$, $\uu_{\kk}-\varphi$ attains a local minimum at $x_{\kk}$ and $D\varphi(x_{\kk})\to b$.
Suppose that 
\begin{eqnarray}\label{35}
\mf{F}_{\infty}(A)<0 \ \Longrightarrow \ \mf{F}_{\kk}(A)<0 \ \ \mbox{for} \ \ \kk\in \N \ \ \mbox{large enough}.
\end{eqnarray}
At this stage, we distinguish two cases according to the behavior of the sequence $\{\bar{\xi}_{\kk}\}$.
\subsubsection*{Case 1: $\{\xi_{\kk}\}$ does not have a convergent subsequence} In this case, up to extract a (non relabelled) subsequence, we have that
\begin{eqnarray}\label{39}
\snr{\xi_{\kk}}\to \infty.
\end{eqnarray}
In the light of \eqref{37}, \eqref{35} and \eqref{39}, up to take $\kk\in \N$ sufficiently large and then relabel, we can assume that
\begin{flalign}\label{38}
\sup_{\kk\in \N}\snr{D\varphi(x_{\kk})}\le 2\left(\snr{b}+1\right),\  \snr{\bar{\xi}_{\kk}}>4\left(\snr{b}+1\right) \ \Longrightarrow \snr{\bar{\xi}_{\kk}+D\varphi(x_{\kk})}\ge 2(\snr{b}+1).
\end{flalign}
Notice that \eqref{35} yields 
\begin{flalign}\label{71}
&\max\left\{\mf{F}_{\kk}(A),\mf{H}_{q;\kk}(x,D\varphi(x_{\kk});\bar{\xi}_{\kk})\mf{F}_{\kk}(A),\mf{H}_{s;\kk}(x,D\varphi(x_{\kk});\bar{\xi}_{\kk})\mf{F}_{\kk}(A)\right\}\nonumber \\
&\qquad \qquad\qquad\qquad= \min\left\{1,\mf{H}_{q;\kk}(x,D\varphi(x_{\kk});\bar{\xi}_{\kk}),\mf{H}_{s;\kk}(x,D\varphi(x_{\kk});\bar{\xi}_{\kk})\right\}\mf{F}_{\kk}(A).
\end{flalign}

As $\uu_{\kk}$ is a $\kk^{-1}$-normalized viscosity supersolution of \eqref{kk2} we get
\begin{eqnarray*}
\mf{F}_{\kk}(A)&\stackrel{\eqref{35}_{2}}{\ge}& -\frac{\mf{C}_{\kk}}{\min\left\{1,\mf{H}_{q;\kk}(x,D\varphi(x_{\kk});\bar{\xi}_{\kk}),\mf{H}_{s;\kk}(x,D\varphi(x_{\kk});\bar{\xi}_{\kk})\right\}}\nonumber \\
&\stackrel{\eqref{38}}{\ge}&-\frac{\kk^{-1}}{\min\{1,2^{p^{+}}(\snr{b}+1)^{p^{+}},2^{p_{-}}(\snr{b}+1)^{p_{-}}\}}.
\end{eqnarray*}
Passing to the limit as $\kk\to \infty$ in the previous display and using \eqref{39} we obtain a contradiction to \eqref{35}.
\subsubsection*{Case 2: $\{\bar{\xi_{\kk}}\}$ admits a convergent subsequence} Up to extract a non relabelled subsequence, we may assume that $\bar{\xi}_{\kk}\to \bar{\xi}_{\infty}$. We first consider the case $\snr{\bar{\xi}_{\infty}+b}>0$, which means that, up to select $\kk\in \N$ large enough and then relabel, 
\begin{eqnarray}\label{40}
\snr{\bar{\xi}_{\kk}+D\varphi(x_{\kk})}\ge \frac{1}{4}\snr{\bar{\xi}_{\infty}+b}>0
\end{eqnarray}
holds true, so from $\eqref{35}_{2}$ and \eqref{kk2} we obtain
\begin{eqnarray*}
\mf{F}_{\kk}(A)&\stackrel{\eqref{71},\eqref{40}}{\ge}& -\frac{\mf{C}_{\kk}}{\min\left\{1,\mf{H}_{q;\kk}(x,4(\bar{\xi}_{\infty}+b);\bar{\xi}_{\infty}),\mf{H}_{s;\kk}(x,\bar{\xi}_{\infty}+b;\bar{\xi}_{\infty})\right\}}\nonumber \\
&\ge& -\frac{4^{p^{+}+p_{-}+1}\kk^{-1}}{\min\{1,\snr{\bar{\xi}_{\infty}+b}^{p^{+}},\snr{\bar{\xi}_{\infty}+b}^{p_{-}}\}}.
\end{eqnarray*}
Sending $\kk\to \infty$ in the above display we contradict \eqref{35}.\\
At this point, we only need to take care of the occurrence $\snr{\bar{\xi}_{\infty}+b}=0$. By \eqref{35} and ellipticity, we deduce that $A$ has at least one positive eigenvalue. Let $\Sigma_{0}$ be the direct sum of all the eigensubspaces corresponding to nonnegative eigenvalues of $A$ and $\Pi_{0}(\cdot)$ be the orthogonal projection over $\Sigma_{0}$. Since $\uu_{\infty}-\varphi$ has a local strict minimum in $x_{0}$, by \eqref{32} the function $$\varphi_{\delta}(x):=\varphi(x)+\delta\snr{\Pi_{0}(x-x_{0})}$$ touches
$\uu_{\infty}$ from below in a point $\hat{x}_{0}$ close to $x_{0}$ for $\delta>0$ sufficiently small. We are then lead to consider two possible occurrences: $\snr{\Pi_{0}(\hat{x}_{0}-x_{0})}=0$ and $\snr{\Pi_{0}(\hat{x}_{0}-x_{0})}>0$. If $\snr{\Pi_{0}(\hat{x}_{0}-x_{0})}=0$, then
\begin{eqnarray*}
\snr{\Pi_{0}(\hat{x}_{0}-x_{0})}=\max_{e\in \mathbb{S}^{n-1}}e\cdot \Pi_{0}(\hat{x}_{0}-x_{0})=\min_{e\in \mathbb{S}^{n-1}}e\cdot \Pi_{0}(\hat{x}_{0}-x_{0}),
\end{eqnarray*}
which means that the map $$\hat{\varphi}_{\delta}(x)=\varphi(x)+\delta e\cdot\Pi_{0}(x-x_{0})$$ touches $\uu_{\infty}$ from below in $\hat{x}_{0}$ for all $e\in \mathbb{S}^{n-1}$. This last fact, \eqref{32} and standard stability results, cf. \cite[Lemma 5]{ka} yield that $\hat{\varphi}_{\delta}(\cdot)$ touches $\uu_{\kk}$ from below in $\hat{x}_{\kk}\to \hat{x}_{0}$. The uniformity prescribed by \eqref{32} guarantees that $\delta$ does not depend on $\kk$. A direct computation shows that $D(e\cdot \Pi_{0}(x-x_{0}))=\Pi_{0}(e)$ and $D^{2}(e\cdot \Pi_{0}(x-x_{0}))=0$. Moreover, it is
\begin{flalign}\label{41}
e\in \Sigma_{0}\cap \mathbb{S}^{n-1}\ \Rightarrow \ \Pi_{0}(e)=e\quad \mbox{and}\quad e\in  \Sigma_{0}^{\perp}\cap \mathbb{S}^{n-1}\ \Rightarrow \ \Pi_{0}(e)=0,
\end{flalign}
where $\Sigma_{0}^{\perp}$ is the subspace orthogonal to $\Sigma_{0}$. We claim that 
\begin{flalign}\label{42}
\mbox{there is} \ \ \hat{e}\in \mathbb{S}^{n-1} \ \ \mbox{so that} \ \ \snr{D\varphi(\hat{x}_{0})+\bar{\xi}_{\infty}+\Pi_{0}(\hat{e})}>0.
\end{flalign}
In fact, if $\snr{D\varphi(\hat{x}_{0})+\bar{\xi}_{\infty}}=0$ we pick any $\hat{e}\in \mathbb{S}^{n-1}\cap \Sigma$ (which exists as $\Sigma\not =\emptyset$ because of the previous considerations on the eigenvalues of $A$) and use $\eqref{41}_{1}$; while if $\snr{D\varphi(\hat{x}_{0})+\bar{\xi}_{\infty}}>0$ and $\Sigma^{\perp}\not =\emptyset$, we fix $\hat{e}\in \Sigma^{\perp}\cap \mathbb{S}^{n-1}$ and exploit $\eqref{41}_{2}$ and if $\snr{D\varphi(\hat{x}_{0})+\bar{\xi}_{\infty}}>0$ and $\Sigma^{\perp} =\emptyset$, i.e. $\Sigma\equiv \mathbb{R}^{n}$ and $\Pi_{0}(\cdot)\equiv \mathbf{I}$, we let $\hat{e}:=\frac{D\varphi(\hat{x}_{0})+\bar{\xi}_{\infty}}{\snr{D\varphi(\hat{x}_{0})+\bar{\xi}_{\infty}}}$, thus
\begin{eqnarray*}
\snr{D\varphi(\hat{x}_{0})+\bar{\xi}_{\infty}+\Pi_{0}(\hat{e})}=\snr{D\varphi(\hat{x}_{0})+\bar{\xi}_{\infty}+\Pi_{0}(\hat{e})}+1>1.
\end{eqnarray*}
Once \eqref{42} has been established, we can take $\kk\in \N$ sufficiently large to assure that 
\begin{eqnarray}\label{43.1}
\snr{D\varphi(\hat{x}_{\kk})+\bar{\xi}_{\kk}+\Pi_{0}(\hat{e})}\ge \frac{1}{4}\snr{D\varphi(\hat{x}_{0})+\bar{\xi}_{\infty}+\Pi_{0}(\hat{e})}\stackrel{\eqref{42}}{>0},
\end{eqnarray}
recall $\eqref{35}_{2}$ and use \eqref{kk2} to conclude with
\begin{eqnarray*}
\mf{F}_{\kk}(A)
&\stackrel{\eqref{71},\eqref{43.1}}{\ge}&-\frac{4^{p^{+}+p_{-}+1}\kk^{-1}}{\min\left\{1,\snr{D\varphi(\hat{x}_{0})+\bar{\xi}_{\infty}+\Pi_{0}(\hat{e})}^{p^{+}},\snr{D\varphi(\hat{x}_{0})+\bar{\xi}_{\infty}+\Pi_{0}(\hat{e})}^{p_{-}}\right\}}.
\end{eqnarray*}
As $\kk\to \infty$ in the above display, we obtain a contradiction to \eqref{35}.\\
On the other hand, if $\snr{\Pi_{0}(\hat{x}_{0}-x_{0})}>0$, we still have that $\varphi_{\delta}(\cdot)$ touches $\uu_{\infty}$ from below in $\hat{x}_{0}$ as above, so, by \eqref{32} and standard stability results \cite[Lemma 5]{ka} we have that $\varphi_{\delta}(\cdot)$ touches from below $\uu_{\kk}$ in $\hat{x}_{\kk}$ for some points $\hat{x}_{\kk}\to \hat{x}_{0}$. We remark that by \eqref{32}, $\delta$ does not depend on $\kk$. Being $\snr{\Pi_{0}(\hat{x}_{0}-x_{0})}>0$, it is also $\snr{\Pi_{0}(\hat{x}_{\kk}-x_{0})}>0$ for $\kk\in \N$ sufficiently large, so the map $x\mapsto \snr{\Pi_{0}(x-x_{0})}$ is smooth and convex in a neighborhood of $\hat{x}_{\kk}$. As $\Pi_{0}(\cdot-x_{0})$ is a projector, there holds
\begin{flalign}\label{43}
\Pi_{0}(x-x_{0})D\Pi_{0}(x-x_{0})=\Pi_{0}(x-x_{0})\quad \mbox{and}\quad D^{2}\snr{\Pi_{0}(x-x_{0})} \ \ \mbox{is nonnegative definite}.
\end{flalign}
Recall that we were assuming that $\snr{\bar{\xi}_{\infty}+b}=0$, so using the very definition of $\Sigma_{0}$ we have
\begin{eqnarray*}
\left| \ \bar{\xi}_{\infty}+D\varphi(\hat{x}_{0})+\delta\frac{\Pi_{0}(\hat{x}_{0}-x_{0})}{\snr{\Pi_{0}(\hat{x}_{0}-x_{0})}}\ \right|^{2}&=&\left| \ A(\hat{x}_{0}-x_{0})+\delta\frac{\Pi_{0}(\hat{x}_{0}-x_{0})}{\snr{\Pi_{0}(\hat{x}_{0}-x_{0})}}\ \right|^{2}\nonumber \\
&=&\snr{A(\hat{x}_{0}-x_{0})}^{2}+\delta^{2}+2\delta A(\hat{x}_{0}-x_{0})\cdot\frac{\Pi_{0}(\hat{x}_{0}-x_{0})}{\snr{\Pi_{0}(\hat{x}_{0}-x_{0})}}\nonumber \\
&=&\snr{A(\hat{x}_{0}-x_{0})}^{2}+\delta^{2}+2\delta A\Pi_{0}(\hat{x}_{0}-x_{0})\cdot\frac{\Pi_{0}(\hat{x}_{0}-x_{0})}{\snr{\Pi_{0}(\hat{x}_{0}-x_{0})}}\nonumber \\
&\ge&\snr{A(\hat{x}_{0}-x_{0})}^{2}+\delta^{2}\ge \delta^{2},
\end{eqnarray*}
thus
\begin{eqnarray*}
\left| \ \bar{\xi}_{\infty}+D\varphi(\hat{x}_{0})+\delta\frac{\Pi_{0}(\hat{x}_{0}-x_{0})}{\snr{\Pi_{0}(\hat{x}_{0}-x_{0})}}\ \right|\ge \delta,
\end{eqnarray*}
therefore, for $\kk\in \N$ large enough we have
\begin{eqnarray}\label{50}
\left| \ \bar{\xi}_{\kk}+D\varphi(\hat{x}_{\kk})+\delta\frac{\Pi_{0}(\hat{x}_{\kk}-x_{0})}{\snr{\Pi_{0}(\hat{x}_{\kk}-x_{0})}}\ \right|\ge \frac{\delta}{4}.
\end{eqnarray}
We can then use \eqref{35}$_{2}$ and that $\uu_{\kk}$ is a $\kk^{-1}$-normalized viscosity supersolution of \eqref{kk2} to get
\begin{eqnarray*}
\mf{F}_{\kk}(A)&\stackrel{\eqref{ell}}{\ge}&\mf{F}_{\kk}(A+D^{2}\snr{\Pi_{0}(\hat{x}_{\kk}-x_{0})})+\lambda\texttt{tr}(D^{2}\snr{\Pi_{0}(\hat{x}_{\kk}-x_{0})})\nonumber \\
&\stackrel{\eqref{43}_{2}}{\ge}&-\frac{\mf{C}_{\kk}}{\min\left\{1,\left|\bar{\xi}_{\kk}+D\varphi(\hat{x}_{\kk})+\frac{\Pi_{0}(\hat{x}_{\kk}-x_{0})}{\snr{\Pi_{0}(\hat{x}_{\kk}-x_{0})}}\right|^{p^{+}},\left|\bar{\xi}_{\kk}+D\varphi(\hat{x}_{\kk})+\frac{\Pi_{0}(\hat{x}_{\kk}-x_{0})}{\snr{\Pi_{0}(\hat{x}_{\kk}-x_{0})}}\right|^{p^{-}}\right\}}\nonumber \\
&\stackrel{\eqref{50}}{\ge}&-\frac{4^{p^{+}+p_{-}+1}\kk^{-1}}{\min\{1,\delta^{p^{+}},\delta^{p_{-}}\}}.
\end{eqnarray*}
Sending $\kk\to \infty$ above we obtain a contradiction to \eqref{35}. \\\\
Combining \emph{Case 1} and \emph{Case 2} we can conclude that $\mf{F}_{\infty}(A)\ge 0$, so $\uu_{\infty}$ is a supersolution of \eqref{34} in $B_{1}(0)$. To show that $\uu_{\infty}$ is also a subsolution to \eqref{34}, we only observe that this is equivalent to prove that $\tilde{\uu}_{\infty}:=-\uu_{\infty}$ is a supersolution of equation
$$
\tilde{\mf{F}}_{\infty}(D^{2}\ti{\uu}_{\infty})=0\qquad \mbox{in} \ \ B_{1}(0),
$$
where we set $\tilde{\mf{F}}_{\infty}(M):=-\mf{F}_{\infty}(-M)$, which is uniformly $(\lambda,\Lambda)$-elliptic in the sense of \eqref{ell}. Hence, we can apply the whole procedure developed above on $\tilde{\uu}_{\infty}$ and conclude that $\uu_{\infty}$ is a viscosity solution of \eqref{34}. Proposition \ref{rhar} then applies and $\ti{\uu}_{\infty}\in C^{1,\alpha}(B_{1/2}(0))$. In particular is valid \eqref{0331}, which contradicts \eqref{33}$_{2}$, and the proof is complete.
\end{proof}

Lemma \ref{har} essentially determines a certain parameter $\varepsilon_{0}\equiv \varepsilon_{0}(\texttt{data})\in (0,1)$ so that it is possible to build a tangential path connecting $\varepsilon_{0}$-normalized viscosity solution of \eqref{d.1} to viscosity solutions of a homogeneous limiting profile for which the Krylov-Safonov regularity theory is available. At this stage, we need to transfer such regularity from the limiting homogeneous problem to viscosity solutions of \eqref{eq}. In this perspective, we establish an oscillation control at discrete scales.
\begin{lemma}\label{disc}
Assume \texttt{set} and let $\varepsilon_{0}\equiv \varepsilon_{0}(\texttt{data})\in (0,1)$ be the smallness parameter determined in Lemma \ref{har}. There are $\sigma\equiv \sigma(n,\lambda,\Lambda)\in (0,1)$ and $\alpha_{0}\equiv \alpha_{0}(n,\lambda,\Lambda,p^{+},p_{-})\in (0,1)$ so that if $\uu\in C(B_{1}(0))$ is a $\varepsilon_{0}$-normalized viscosity solution of equation \eqref{eq}, then for any $\kk\in N$ it is possible to find $\bar{\xi}_{\kk}\in\mathbb{R}^{n}$ so that
\begin{eqnarray}\label{51}
\osc_{B_{\sigma_{\kk}}(0)}\left(\uu-\xi_{\kk}\cdot x\right)\le \sigma^{\kk(1+\alpha_{0})}. 
\end{eqnarray}
\end{lemma}
\begin{proof}
Let $\sigma\equiv \sigma(n,\lambda,\Lambda)$ be the one in \eqref{delta} and 
\begin{eqnarray}\label{gamma0}
\alpha_{0}\in \left(0,\min\left\{\alpha,\frac{1}{\max\{p^{+},p_{-}\}+1},\frac{\log(2)}{-\log(\sigma)}\right\}\right),
\end{eqnarray}
where $\alpha\equiv\alpha(n,\lambda,\Lambda)\in (0,1)$ is the H\"older continuity exponent provided by Proposition \ref{rhar}. A direct consequence of the choice made in \eqref{gamma0} is 
\begin{eqnarray}\label{54}
\sigma^{\alpha_{0}}>\frac{1}{2}.
\end{eqnarray}
Now we look back at the construction developed in Section \ref{sr} and fix a scaling parameter $\tau_{0}$ equal to $\varepsilon_{0}^{\frac{1}{2}}$, where $\varepsilon_{0}\equiv \varepsilon_{0}(\texttt{data})$ is the one provided by Lemma \ref{har}. In this way we determine the dependency $\tau_{0}\equiv \tau_{0}(\texttt{data})$ and remove the ambiguity raised in Remark \ref{rema} as now it is $$\nr{\mf{a}}_{L^{\infty}(B_{1}(0))}+\nr{\mf{b}}_{L^{\infty}(B_{1}(0))}\le c(\texttt{data},\nr{a}_{L^{\infty}(\Omega)},\nr{b}_{L^{\infty}(\Omega)},\nr{u}_{L^{\infty}(\Omega)},\nr{f}_{L^{\infty}(\Omega)}).$$ 
Let $\uu\in C(B_{1}(0))$ be a $\varepsilon_{0}$-normalized viscosity solution of equation \eqref{eq} in the sense of Definition \ref{def3} and of Section \ref{sr}, which means that $\uu$ is a $\varepsilon_{0}$-normalized viscosity subsolution/supersolution of \eqref{d.1t}/\eqref{d.2t}. With $\kk\in \N\cup\{0\}$, we define $\sigma_{\kk}:=\sigma^{\kk}$ and start an induction argument to show that \eqref{51} holds for all $\kk\in \N\cup\{0\}$.
\subsubsection*{Basic step - $\kk=0$} By $\eqref{30}_{2}$ we see that \eqref{51} holds with $\bar{\xi}_{0}=0$. In fact it is
\begin{eqnarray*}
\osc_{B_{\sigma_{0}(0)}}\left(\uu-\bar{\xi}_{0}\cdot x\right)=\osc_{B_{1}(0)}\uu\stackrel{\eqref{30}_{2}}{\le}1.
\end{eqnarray*}
\subsubsection*{Induction step} Assume that there exists $\bar{\xi}_{\kk}\in \mathbb{R}^{n}$ satisfying \eqref{51} and define
\begin{eqnarray*}
\uu_{\kk}(x):=\sigma_{\kk}^{-(1+\alpha_{0})}\left[\uu(\sigma_{\kk}x)-\sigma_{\kk}\bar{\xi}_{\kk}\cdot x\right].
\end{eqnarray*}
Recalling Definition \ref{def3}, a straightforward computation shows that $\uu_{\kk}$ is a viscosity subsolution of 
\begin{flalign}\label{62}
\min\left\{\mf{F}_{\kk}(D^{2}\uu_{\kk}),\mf{H}_{q;\kk}(x,D\uu_{\kk};\ti{\xi}_{\kk})\mf{F}_{\kk}(D^{2}\uu_{\kk}),\mf{H}_{s;\kk}(x,D\uu_{\kk};\ti{\xi}_{\kk})\mf{F}_{\kk}(D^{2}\uu_{\kk})\right\}=\mf{C}_{\kk}\quad \mbox{in} \ \ B_{1}(0)
\end{flalign}
and a viscosity supersolution to
\begin{flalign}\label{63}
\max\left\{\mf{F}_{\kk}(D^{2}\uu_{\kk}),\mf{H}_{q;\kk}(x,D\uu_{\kk};\ti{\xi}_{\kk})\mf{F}_{\kk}(D^{2}\uu_{\kk}),\mf{H}_{s;\kk}(x,D\uu_{\kk};\ti{\xi}_{\kk})\mf{F}_{\kk}(D^{2}\uu_{\kk})\right\}=-\mf{C}_{\kk}\quad \mbox{in} \ \ B_{1}(0),
\end{flalign}
where it is
\begin{flalign*}
&\ti{\xi}_{\kk}:=\sigma_{\kk}^{-\alpha_{0}}\bar{\xi}_{\kk},\qquad \mf{a}_{\kk}(x):=\sigma_{\kk}^{\alpha_{0}(q-p^{+})}\mf{a}(\sigma_{\kk}x),\qquad \mf{b}_{\kk}(x):=\sigma_{\kk}^{\alpha_{0}(s-p_{-})}\mf{b}(\sigma_{\kk}x)\nonumber \\
&\mf{F}_{\kk}(M):=\sigma_{\kk}^{1-\alpha_{0}}\mf{F}(\sigma_{\kk}^{\alpha_{0}-1}M),\qquad \mf{C}_{\kk}:=\sigma_{\kk}^{1-\alpha_{0}(\max\{p^{+},p_{-}\}+1)}\mf{C},
\end{flalign*}
$\mf{H}_{q;\kk}(\cdot)$, $\mf{H}_{s;\kk}(\cdot)$ are the same defined in the proof of Lemma \ref{har} and $\mf{C}$ is the constant derived in Section \ref{sr} corresponding to the scaling parameter $\tau_{0}$ fixed before. Notice that by construction $\mf{F}_{\kk}(\cdot)$ satisfies \eqref{assf} uniformly in $\kk$ and because of the choice of $\tau_{0}\equiv \tau_{0}(\texttt{data})$ made above, we have 
\begin{eqnarray*}
\mf{C}_{\kk}\le \sigma_{\kk}^{1-\alpha_{0}(\max\{p^{+},p_{-}\}+1)}\mf{C}\stackrel{\eqref{30}_{2},\eqref{gamma0}}{\le}\varepsilon_{0}.
\end{eqnarray*}
Furthermore, the induction assumption assures that
\begin{eqnarray}\label{61}
\osc_{B_{1}(0)}\uu_{\kk}=\sigma_{\kk}^{-(1+\alpha_{0})}\osc_{B_{\sigma_{\kk}}(0)}\left(\uu-\bar{\xi}_{\kk}\cdot x\right)\stackrel{\eqref{51}}{\le}1,
\end{eqnarray}
and, as $\uu(0)=0$, cf. Section \ref{sr}, it is also $\uu_{\kk}(0)=0$ so by \eqref{61} we have $\nr{\uu_{\kk}}_{L^{\infty}(B_{1}(0))}\le 1$. Therefore we see that $\uu_{\kk}$ is actually a $\varepsilon_{0}$-normalized viscosity subsolution/supersolution of \eqref{62}/\eqref{63}, thus all the assumptions of Lemma \ref{har} are verified, so there is $\ti{\xi}_{\kk+1}\in \mathbb{R}^{n}$ so that
\begin{eqnarray*}
\osc_{B_{\sigma}(0)}\left(\uu_{\kk}-\ti{\xi}_{\kk+1}\cdot x\right)\le \frac{\sigma}{2}.
\end{eqnarray*}
Setting $\bar{\xi}_{\kk+1}:=\bar{\xi}_{\kk}+\sigma_{\kk}^{\alpha_{0}}\ti{\xi}_{\kk+1}$, we can rewrite the content of the previous display as
\begin{eqnarray*}
\sigma_{\kk}^{-(1+\alpha_{0})}\osc_{B_{\sigma_{\kk+1}}(0)}\left(\uu-\bar{\xi}_{\kk+1}\cdot x\right)\le \frac{\sigma}{2} \ \Longrightarrow \ \osc_{B_{\sigma_{\kk+1}}(0)}\left(\uu-\bar{\xi}_{\kk+1}\cdot x\right)\stackrel{\eqref{54}}{\le} \sigma_{\kk+1}^{1+\alpha_{0}}
\end{eqnarray*}
and the proof is complete.
\end{proof}
Now we are ready to prove Theorem \ref{t2}.
\subsection{Proof of Theorem \ref{t2}} Let $u\in C(\Omega)$ be a viscosity solution of equation \eqref{eq}. For the parameter $\varepsilon_{0}\equiv \varepsilon_{0}(\texttt{data})\in (0,1)$ provided by Lemma \ref{har}, we follow the scaling process outlined in Section \ref{sr} to turn $u$ into a $\varepsilon_{0}$-normalized viscosity solution of \eqref{eq}. The choice of $\varepsilon_{0}$ assures that the assumptions of Lemma \ref{disc} are satisfied, so \eqref{51} is available to us. Given any $\rr\in (0,1]$, we can find $\kk\in \N\cup\{0\}$ so that $\sigma^{\kk+1}< \rr \le \sigma^{\kk}$. We then estimate
\begin{eqnarray*}
\osc_{B_{\rr}(0)}\left(\uu-\xi_{\kk}\cdot x\right)\le \osc_{B_{\sigma^{\kk}(0)}}\left(\uu-\xi_{\kk}\cdot x\right)\stackrel{\eqref{51}}{\le}\sigma^{\kk(1+\alpha_{0})}\le \sigma^{-(1+\alpha_{0})}\rr^{1+\alpha_{0}}\le c\rr^{1+\alpha_{0}},
\end{eqnarray*}
with $c\equiv c(n,\lambda,\Lambda,p^{+},p_{-})$,
so $\uu$ is $C^{1,\alpha_{0}}$-regular around zero. By standard translation arguments we can prove the same fact in a neighborhood of any $x_{0}\in B_{1/2}(0)$. In particular, we have
\begin{eqnarray*}
[D\uu]_{0,\alpha_{0};B_{1/2}(0)}\le c(n,\lambda,\Lambda,p^{+},p_{-}).
\end{eqnarray*}
Reversing the scaling procedure in Section \ref{sr} and applying the usual covering argument we obtain \eqref{60}, which implies that $u\in C^{1,\alpha_{0}}_{\loc}(\Omega)$ and the proof is complete.
\appendix
\section{H\"older estimates for multi-phase equations with variable exponents}\label{vem}
Let us derive uniform H\"older estimates for continuous viscosity solutions to fully nonlinear elliptic equations of Multi-Phase type with variable exponents. Let $\mu\in [0,1]$ be any number, set for simplicity
$$
\Omega\times \mathbb{R}^{n}\ni (x,z)\mapsto G_{\mu}(x,z):=\left[\ell_{\mu}(z)^{p(x)}+a(x)\ell_{\mu}(z)^{q(x)}+b(x)\ell_{\mu}(z)^{s(x)}\right].
$$
and consider equation
\begin{flalign}\label{mpe}
G_{\mu}(x,Du)\left(\mu u+F(D^{2}u)\right)=f(x)\qquad \mbox{in} \ \ \Omega,
\end{flalign}
where
\begin{eqnarray}\label{pxqxsx}
0\le p(\cdot)\in C(\Omega),\qquad 0\le q(\cdot)\in C(\Omega),\qquad 0\le s(\cdot)\in C(\Omega)
\end{eqnarray}
and assume also \eqref{assf}, \eqref{ab} and \eqref{f}.
\begin{proposition}\label{prophol}
Under assumptions \eqref{assf}, \eqref{ab}, \eqref{f} and \eqref{pxqxsx}, let $u\in C(\Omega)$ be a viscosity solution to the Multi-Phase fully nonlinear equation with variable exponents \eqref{mpe}. Then $u\in C^{0,\beta_{0}}_{\loc}(\Omega)$ for all $\beta_{0}\in (0,1)$. In particular, if $B_{\rr}(z_{0})\Subset \Omega$ is any ball with radius $\rr\in \left(0,\frac{1}{2}\right)$, it holds that
\begin{eqnarray}\label{a.es}
[u]_{0,\beta_{0};B_{\rr/2}(z_{0})}\le c(n,\lambda,\Lambda,\nr{u}_{L^{\infty}(B_{\rr}(z_{0}))},\nr{f}_{L^{\infty}(B_{\rr}(z_{0}))},\rr,\beta_{0}).
\end{eqnarray}
\end{proposition}
\begin{proof}
Let $u\in C(\Omega)$ be a viscosity solution to equation \eqref{mpe} and $B_{\rr}(z_{0})\Subset \Omega$ be any ball with radius $\rr\in \left(0,\frac{1}{2}\right)$. We prove that there are two constants $A_{2}\equiv A_{2}(\rr,\nr{u}_{L^{\infty}(B_{\rr}(z_{0}))})$ and $A_{1}\equiv A_{1}(n,\lambda,\Lambda,p,\rr,\beta_{0},\nr{u}_{L^{\infty}(B_{\rr}(z_{0}))},\nr{f}_{L^{\infty}(B_{\rr}(z_{0}))})$ so that 
\begin{flalign}\label{a0}
\mathcal{M}(x_{0}):=\sup_{x,y\in B_{\rr}(z_{0})}\left(u(x)-u(y)-A_{1}\snr{x-y}^{\beta_{0}}-A_{2}\left(\snr{x-x_{0}}^{2}+\snr{y-x_{0}}^{2}\right)\right)\le 0
\end{flalign}
holds for all $x_{0}\in B_{\rr/2}(z_{0})$. In \eqref{a0}, $\beta_{0}\in (0,1)$ is any (fixed) number. By contradiction, we assume that
\begin{eqnarray}\label{a2}
\textnormal{there exists}\ x_{0}\in B_{\rr/2}(z_{0}) \ \textnormal{such that} \ \mathcal{M}(x_{0})> 0 \ \textnormal{for all positive} \ A_{1},A_{2},
\end{eqnarray}
define quantities
\begin{flalign}\label{a3}
\left\{
\begin{array}{c}
\displaystyle 
\ A_{1}:=\frac{4}{\beta_{0}(1-\beta_{0})}\left[\frac{\nr{f}_{L^{\infty}(B_{\rr}(z_{0}))}+\nr{u}_{L^{\infty}(B_{\rr}(z_{0}))}}{\lambda}+(2A_{2}+1)\left(\frac{\Lambda(n-1)}{\lambda}+1\right)\right]\\[17pt] \displaystyle
\ A_{2}:=64\rr^{-2}\nr{u}_{L^{\infty}(B_{\rr}(z_{0}))}
\end{array}
\right.
\end{flalign}
and consider the auxiliary functions
\begin{eqnarray*}
\begin{cases}
\ \psi(x,y):=A_{1}\snr{x-y}^{\beta_{0}}+A_{2}\left(\snr{x-x_{0}}^{2}+\snr{y-x_{0}}^{2}\right)\\
\ \phi(x,y):=u(x)-u(y)-\psi(x,y).
\end{cases}
\end{eqnarray*}
If $(\bar{x},\bar{y})\in \bar{B}_{\rr}(z_{0})\times \bar{B}_{\rr}(z_{0})$ is a maximum point of $\phi(\cdot)$, via \eqref{a0} we have $\phi(\bar{x},\bar{y})=\mathcal{M}(x_{0})>0$, so
\begin{eqnarray*}
A_{1}\snr{\bar{x}-\bar{y}}^{\beta_{0}}+A_{2}\left(\snr{\bar{x}-x_{0}}^{2}+\snr{\bar{y}-x_{0}}^{2}\right)\le u(\bar{x})-u(\bar{y})\le 2\nr{u}_{L^{\infty}(B_{\rr}(z_{0}))}.
\end{eqnarray*}
Plugging $\eqref{a3}_{2}$ in the above inequality yields that $\bar{x}$, $\bar{y}$ both belong to the interior of $B_{\rr}(z_{0})$, in fact:
\begin{eqnarray*}
\snr{\bar{x}-z_{0}}\le \snr{\bar{x}-x_{0}}+\snr{x_{0}-z_{0}}\le \frac{3\rr}{4}\qquad \mbox{and}\qquad \snr{\bar{y}-z_{0}}\le \snr{\bar{y}-x_{0}}+\snr{x_{0}-z_{0}}\le \frac{3\rr}{4}.
\end{eqnarray*}
Moreover, $\bar{x}\not =\bar{y}$, otherwise $\mathcal{M}(x_{0})=\phi(\bar{x},\bar{y})=0$ and \eqref{a0} would be verified. This last remark shows that $\psi(\cdot)$ is smooth in a small neighborhood of $(\bar{x},\bar{y})$, therefore we can determine vectors
\begin{flalign*}
&\xi_{\bar{x}}:=\partial_{x}\psi(\bar{x},\bar{y})=A_{1}\beta_{0}\snr{\bar{x}-\bar{y}}^{\beta_{0}-1}\frac{\bar{x}-\bar{y}}{\snr{\bar{x}-\bar{y}}}+2A_{2}(\bar{x}-x_{0}),\\
&\xi_{\bar{y}}:=-\partial_{y}\psi(\bar{x},\bar{y})=A_{1}\beta_{0}\snr{\bar{x}-\bar{y}}^{\beta_{0}-1}\frac{\bar{x}-\bar{y}}{\snr{\bar{x}-\bar{y}}}-2A_{2}(\bar{y}-x_{0}).
\end{flalign*}
To summarize, we have that $\phi(\cdot)$ attains its maximum in $(\bar{x},\bar{y})$ inside $B_{\rr}(z_{0})\times B_{\rr}(z_{0})$ and $\phi(\cdot)$ is smooth around $(\bar{x},\bar{y})$, thus Proposition \ref{p1} applies: for any $\iota>0$ we can find a threshold $\hat{\delta}=\hat{\delta}(\iota,\nr{D^{2}\psi})$ such that for all $\delta \in (0,\hat{\delta})$ the couple $(\xi_{\bar{x}},X_{\delta})$ is a limiting subjet of $u$ at $\bar{x}$ and the couple $(\xi_{\bar{y}},Y_{\delta})$ is a limiting superjet of $u$ at $\bar{y}$ and the matrix inequality
\begin{flalign}\label{m2}
\begin{bmatrix}
X_{\delta} & 0 \\ 0 & -Y_{\delta} 
\end{bmatrix}\le \begin{bmatrix}
Z & -Z \\ -Z & Z 
\end{bmatrix}+(2A_{2}+\delta)\mathbf{I}
\end{flalign}
holds, where we set
\begin{flalign*}
Z:=&\left.A_{1}D^{2}(\snr{x-y}^{\beta_{0}})\right|_{(\bar{x},\bar{y})}\nonumber \\
=&A_{1}\left[\frac{\beta_{0}\snr{\bar{x}-\bar{y}}^{\beta_{0}-1}}{\snr{\bar{x}-\bar{y}}}\mathbf{I}+\left(\beta_{0}(1-\beta_{0})\snr{\bar{x}-\bar{y}}^{\beta_{0}-2}-\frac{\beta_{0}\snr{\bar{x}-\bar{y}}^{\beta_{0}-1}}{\snr{\bar{x}-\bar{y}}}\right)\frac{(\bar{x}-\bar{y})\otimes (\bar{x}-\bar{y})}{\snr{\bar{x}-\bar{y}}^{2}}\right].
\end{flalign*}
We fix $\delta\equiv \min\left\{1,\frac{\hat{\delta}}{4}\right\}$ and apply \eqref{m2} to vectors of the form $(z,z)\in \mathbb{R}^{2n}$, to obtain  
\begin{flalign*}
\langle(X_{\delta}-Y_{\delta})z,z \rangle\le (4A_{2}+2)\snr{z}^{2}.
\end{flalign*}
This means that 
\begin{flalign}\label{alle}
\mbox{all the eigenvalues of} \ \ X_{\delta}-Y_{\delta} \ \ \mbox{are less than or equal to} \ \ 2(2A_{2}+1).
\end{flalign}
In particular, applying \eqref{m2} to the vector $\bar{z}:=\left(\frac{\bar{x}-\bar{y}}{\snr{\bar{x}-\bar{y}}},\frac{\bar{y}-\bar{x}}{\snr{\bar{x}-\bar{y}}}\right)$, we get
\begin{flalign*}
\left \langle (X_{\delta}-Y_{\delta})\frac{\bar{x}-\bar{y}}{\snr{\bar{x}-\bar{y}}},\right.&\left.\frac{\bar{x}-\bar{y}}{\snr{\bar{x}-\bar{y}}} \right \rangle\le 2(2A_{2}+1)-4\beta_{0}(1-\beta_{0})A_{1}\snr{\bar{x}-\bar{y}}^{\beta_{0}-2}.
\end{flalign*}
This yields in particular that
\begin{flalign}\label{eneg}
\mbox{at least one eigenvalue of} \ \ X_{\delta}-Y_{\delta} \ \ \mbox{is less than} \ \ 2(2A_{2}+1)-4A_{1}\beta_{0}(1-\beta_{0})\snr{\bar{x}-\bar{y}}^{\beta_{0}-2}.
\end{flalign}
Expanding the expression of $\omega(\cdot)$ in \eqref{eneg} we get
\begin{eqnarray*}
2(2A_{2}+1)+4A_{1}\beta_{0}(\beta_{0}-1)\snr{\bar{x}-\bar{y}}^{\beta_{0}-2}\le 2(2A_{2}+1)-4A_{1}\beta_{0}(1-\beta_{0})\stackrel{\eqref{a3}_{1}}{<}0.
\end{eqnarray*}
where we also used that $\snr{\bar{x}-\bar{y}}\le 1$. This means that at least one eigenvalue of $X_{\delta}-Y_{\delta}$ is negative, thus combining \eqref{m-}$_{2}$, \eqref{alle} and \eqref{eneg} we obtain
\begin{eqnarray}\label{a4}
\mathcal{M}_{\lambda,\Lambda}^{-}(X_{\delta}-Y_{\delta})\ge -2(2A_{2}+1)\left[\Lambda(n-1)+\lambda\right]+4\lambda A_{1}\beta_{0}(1-\beta_{0}).
\end{eqnarray}
With $\xi_{\bar{x}}$, $\xi_{\bar{y}}$ computed before, we recover the viscosity inequalities
\begin{flalign}\label{a5}
\begin{cases}
\ G_{\mu}(\bar{x},\xi_{\bar{x}})\left(\mu u(\bar{x})+F(X_{\delta})\right)\le f(\bar{x})\\
\ G_{\mu}(\bar{y},\xi_{\bar{y}})\left(\mu u(\bar{y})+F(Y_{\delta})\right)\ge f(\bar{y}).
\end{cases}
\end{flalign}
Moreover, a quick computation shows that
\begin{eqnarray}\label{a5.1}
\min\left\{\ell_{\mu}(\xi_{\bar{x}}),\ell_{\mu}(\xi_{\bar{y}})\right\}\ge\min\left\{\snr{\xi_{\bar{x}}},\snr{\xi_{\bar{y}}}\right\}\ge\sqrt{A_{1}\beta_{0}\left(A_{1}\beta_{0}-2A_{2}\right)}\stackrel{\eqref{a3}_{1}}{\ge}1
\end{eqnarray}
and, via ellipticity,
\begin{eqnarray}\label{a7}
F(X_{\delta})\stackrel{\eqref{elll}}{\ge} F(Y_{\delta})+\mathcal{M}^{-}_{\lambda,\Lambda}(X_{\delta}-Y_{\delta}).
\end{eqnarray}
Merging all the previous inequalities, we obtain
\begin{eqnarray*}
\frac{f(\bar{x})}{G_{\mu}(\bar{x},\xi_{\bar{x}})}&\stackrel{\eqref{a5}_{1}}{\ge}&\mu u(\bar{x})+F(X_{\delta})\nonumber \\
&\stackrel{\eqref{a7}}{\ge}&\mu u(\bar{x})+F(Y_{\delta})+\mathcal{M}^{-}_{\lambda,\Lambda}(X_{\delta}-Y_{\delta})\nonumber \\
&\stackrel{\eqref{a4},\eqref{a5}_{2}}{\ge}&\mu(u(\bar{x})-u(\bar{y}))\nonumber \\
&-&2(2A_{2}+1)\left[\Lambda(n-1)+\lambda\right]+4\lambda A_{1}\beta_{0}(1-\beta_{0})+\frac{f(\bar{y})}{G_{\mu}(\bar{y},\xi_{\bar{y}})},
\end{eqnarray*}
so with \eqref{a5.1} we can complete the estimate in the above display as follows:
\begin{eqnarray*}
2\left(\nr{f}_{L^{\infty}(B_{\rr}(z_{0}))}+\nr{u}_{L^{\infty}(B_{\rr}(z_{0}))}\right)&+&2(2A_{2}+1)\left[\Lambda(n-1)+\lambda\right]\nonumber \\
&\ge&\mu(u(\bar{y})-u(\bar{x}))+\frac{f(\bar{x})}{G_{\mu}(\bar{x},\xi_{\bar{x}})}-\frac{f(\bar{y})}{G_{\mu}(\bar{y},\xi_{\bar{y}})}\nonumber \\
&+&2(2A_{2}+1)\left[\Lambda(n-1)+\lambda\right]\nonumber \\
&\ge&4\lambda A_{1}\beta_{0}(1-\beta_{0}),
\end{eqnarray*}
which contradicts the position in $\eqref{a3}_{1}$. This means that there are two positive constants $A_{1}$, $A_{2}$ with the dependencies outlined before so that for all $x_{0}\in B_{\rr/2}(z_{0})$, inequality \eqref{a0} is verified, which in particular yields that $u\in C^{0,\beta_{0}}(B_{\rr/2}(z_{0}))$ for all $\beta_{0}\in (0,1)$. The arbitrariety of $B_{\rr}(z_{0})$ and a standard covering argument render that $u\in C^{0,\beta_{0}}_{\loc}(\Omega)$ for all $\beta_{0}\in (0,1)$ and the proof is complete. 
\end{proof}
\begin{remark}\label{r1}
\emph{Notice that the constant appearing in \eqref{a.es} does not depend on $\mu\in [0,1]$ nor on the moduli of continuity of $a(\cdot)$, $b(\cdot)$, $p(\cdot)$, $q(\cdot)$, $s(\cdot)$.}
\end{remark}

\end{document}